\documentclass{article}
\usepackage[a4paper]{geometry}
\usepackage{amssymb}
\usepackage{systeme}
\usepackage[english]{babel}
\usepackage{color}
\usepackage{amsmath}
\usepackage{amsthm}
\usepackage{graphicx}
\usepackage{geometry}
\usepackage{mathtools}
\usepackage{comment}
\usepackage[normalem]{ulem}
\allowdisplaybreaks

\usepackage[urlcolor=black]{hyperref}
\hypersetup{
    colorlinks=true, 
    linktoc=all,     
    linkcolor=black,  
    citecolor=black
}

\theoremstyle{definition}
\newtheorem{theorem}{Theorem}[section]
\newtheorem{definition}[theorem]{Definition}

\newtheorem{corollary}[theorem]{Corollary}
\newtheorem{lemma}[theorem]{Lemma}
\newtheorem{remark}[theorem]{Remark}
\newtheorem{proposition}[theorem]{Proposition}

\title{Towards resurgence of Joyce structures}
\date{\small School of Mathematics and Statistics\\ University of Sheffield\\
Hounsfield Road, Sheffield S3 7RH, United Kingdom}
\author{Iv\'an Tulli}
\numberwithin{equation}{section}

\begin{document}
\emergencystretch 3em
\maketitle
\begin{abstract}
    Given a Joyce structure, we show that the associated $\mathbb{C}^*$-family of non-linear connections $\mathcal{A}^{\epsilon}$ can be gauged to a standard form $\mathcal{A}^{\epsilon,\text{st}}$ by a gauge transformation $\hat{g}$, formal in $\epsilon$. We show that the corresponding infinitesimal gauge transformation $\dot{g}=\log(\hat{g})$ has a convergent Borel transform, provided $\dot{g}$ vanishes on the base of the Joyce structure.  This establishes the first step in showing that such a $\dot{g}$ is resurgent. We also use $\hat{g}$ to produce formal twistor Darboux coordinates for the complex hyperk\"{a}hler structure associated to the Joyce structure, and show a similar result about convergence of the Borel transform of the formal twistor Darboux coordinates.  
\end{abstract}

\tableofcontents
\section{Introduction}

The topic of this paper is motivated by two related lines of work: the work of T. Bridgeland related to Donaldson-Thomas (DT) invariants and Joyce structures \cite{BridgeJoyce, TwistorJoyce}, and the work of M. Kontsevich and Y. Soibelman relating resurgent series to analytic stability data \cite{KSresurgence}.\\

The notion of Joyce structure was introduced by T. Bridgeland \cite{BridgeJoyce} in the context of DT invariants of a 3d Calabi-Yau category $\mathfrak{X}$ and its associated space of stability conditions $M=\text{Stab}(\mathfrak{X})$ \cite{BridgeStab}. It is conjectured to describe a geometric structure on the tangent bundle $TM$, whose construction involves solving a family (parametrized by $TM$) of non-linear Riemann-Hilbert problems determined by the DT invariants \cite{varBPS}. Several examples of Joyce structures have been studied in different contexts, including integrable systems and topological string theory \cite{BridgeJoyce, BridgeFab, JoyceQuad, BridgelandOsc, DunakTimJoyce, TimJoyce, Alexandrov_2021, Alexandrov_2021_con}.  \\

While Joyce structures were introduced in the aforementioned context, they can be formulated over holomorphic symplectic manifolds (or more generally holomorphic Poisson manifolds \cite{BridgeJoyce}). Roughly speaking, a Joyce structure over a holomorphic symplectic manifold $(M,\Omega)$ consists of a family of non-linear connections $\mathcal{A}^{\epsilon}$ on the canonical projection $\pi:TM \to M$, parametrized by $\epsilon \in \mathbb{C}^{*}=\mathbb{C}-\{0\}$. The family $\mathcal{A}^{\epsilon}$ must be flat and symplectic, and has the form 
\begin{equation}\label{Joyceconn}
    \mathcal{A}^{\epsilon}=h+\epsilon^{-1}v,
\end{equation}
where $h:\pi^*(TM)\to T(TM)$ is a fixed connection, and $v:\pi^*(TM)\to \text{Ker}(\pi_*)$ is the canonical identification between $\pi^*(TM)$ and the vertical bundle $\text{Ker}(\pi_*)$. There are additional structures and conditions, but we will omit them for now (see Section \ref{Joyce_section} for more details).\\

One of the objectives of this work is to study the formal classification of Joyce structures, where we allow for gauge transformations $\hat{g}$ that are given by formal series in $\epsilon$. In order to explain more clearly what we mean by formal classification of Joyce structures, let us introduce a few more notions: 
\begin{itemize}
\item On one hand, it is convenient to think of the family of connections $\mathcal{A}^{\epsilon}$ as a single relative connection $\mathcal{A}$ on $p:TM\times \mathbb{C}\to M\times \mathbb{C}$ with a simple pole at $\epsilon=0$, where $p$ acts by $\pi$ on the first factor and as the identity on the second factor. The connection is relative with respect to the canonical projection $M\times \mathbb{C}\to \mathbb{C}$ since it does not lift any vectors tangent to the $\mathbb{C}$-factor. One can then replace the $\mathbb{C}$-factors by the formal disc $\widehat{D}=\text{Spf}(\mathbb{C}[[\epsilon]])$ and obtain a meromorphic relative connection on $\widehat{p}: TM\times \widehat{D}\to M\times \widehat{D}$, where we think of the spaces $TM\times \widehat{D}$ and $M\times \widehat{D}$ as formal analytic spaces. In this setting, we will consider the group $\widehat{\mathcal{G}}$ of gauge transformations of $\widehat{p}: TM\times \widehat{D}\to M\times \widehat{D}$ extending the identity at $\epsilon=0$, and its action on connections on $\widehat{p}$ (see Section \ref{formal_setting} for more details). 

\item On the other hand, let us introduce an additional structure on $(M,\Omega)$ that is present in all Joyce structures. This additional structure is a flat torsion-free linear connection $\nabla$ on $M$, such that $\nabla \Omega =0$. With respect to this additional structure the expression \eqref{Joyceconn} of $\mathcal{A}^{\epsilon}$ can be refined to 
\begin{equation*}
    \mathcal{A}=\mathcal{H}+\omega+\epsilon^{-1}v\,,
\end{equation*}
where $\mathcal{H}:\pi^*(TM)\to T(TM)$ is the horizontal lift induced by $\nabla$, and $\omega:\pi^*(TM)\to \text{Ker}(\pi_*)$ is the vertical difference $\omega:=h-\mathcal{H}$. \\
\end{itemize}

In Theorem \ref{mt1}, we show that we can always find a formal gauge transformation $\hat{g}\in \widehat{\mathcal{G}}$, gauging $\mathcal{A}$ to the standard form
\begin{equation*}
    \mathcal{A}^{\text{st}}=\mathcal{H}+\epsilon^{-1}v\,.
\end{equation*}
Furthermore, $\hat{g}$ is unique if we impose that it restricts to the identity on the zero section $M\subset TM$ (or equivalently, the infinitesimal gauge transformation $\dot{g}=\log(\hat{g})$ vanishes on $M$). \\

We remark that this is very much related to the work of M. Kontsevich and Y. Soibelman \cite[Section 5.3]{KSresurgence}, where they study a class of non-linear formal connections called almost standard connections, and show they are formally gauge equivalent to a standard form, called standard connections.  However, our approach does not use the structure of algebraic vector fields on algebraic tori, and uses only the relative connections coming from \eqref{Joyceconn}, rather than the full formal connections also lifting tangent vectors in the $\epsilon$-direction.  \\

In subsequent work by T. Bridgeland and I. Strachan \cite{BriStr}, it was shown that a Joyce structure encodes a complex hyperk\"{a}hler ($\mathbb{C}$-HK) structure on $N=TM$. This is a holomorphic analogue of a usual hyperk\"{a}hler structure, where all tensors (i.e. the metric and the three complex structures) are now holomorphic with respect to the complex structure of $N$ coming from $M$. Since a Joyce structure has an associated $\mathbb{C}$-HK structure, it also has an associated twistor space $p:\mathcal{Z}\to \mathbb{P}^1$ (see \cite{TwistorJoyce}). In fact, when Joyce structures are built from solutions of non-linear Riemann-Hilbert problems associated to DT-invariants, the (logarithms of the) solutions are expected to descend to Darboux coordinates for the $\mathcal{O}(2)$-twisted family of relative holomorphic symplectic forms $\varpi$ of the twistor space $\mathcal{Z}$. In Proposition \ref{flatcoordsprop}, we construct formal twistor Darboux coordinates for $\varpi$ from the formal gauge transformation $\hat{g}$ gauging $\mathcal{A}^{\text{st}}$ to $\mathcal{A}$.\\

Now consider a formal gauge transformation $\hat{g}$ gauging $\mathcal{A}^{\text{st}}$ to $\mathcal{A}$, and denote by $\dot{g}=\log(\widehat{g})$ the corresponding infinitesimal gauge transformation. More precisely, $\dot{g}$ is a formal series
\begin{equation}\label{dotgintro}
    \dot{g}=\sum_{k=1}^{\infty}\dot{g}_k\epsilon^k
\end{equation}
where $\dot{g}_k$ are vector fields on $TM$ vertical with respect to $\pi:TM\to M$ (i.e. sections of $\text{Ker}(\pi_*)\to TM$). The other main purpose of this work is to show the first step towards establishing that $\dot{g}$ is resurgent. The series $\dot{g}$ is resurgent if:
\begin{itemize}
    \item The Borel transform
\begin{equation}\label{btintro}
    \mathcal{B}[\dot{g}](\xi):=\sum_{k=1}^{\infty}\frac{\dot{g}_k}{(k-1)!}\xi^{k-1}
\end{equation}
converges in a neighbourhood of $\xi=0$.
\item There is some closed discrete subset $S\subset \mathbb{C}$ such that $\mathcal{B}[\dot{g}](\xi)$ admits analytic continuations in $\xi$ along paths in $\mathbb{C}-S$ starting from the neighbourhood of convergence of \eqref{btintro}.
\end{itemize} Under mild assumptions on the growth of the analytic continuations of $\mathcal{B}[\dot{g}](\xi)$ as $\xi\to \infty$, one can consider the Borel summation of $\dot{g}$ along a ray $\rho=\mathbb{R}_{>0}e^{\mathrm{i}s}$ avoiding $S$, given by the Laplace transform
\begin{equation}\label{BorelSum}
    \dot{g}_{\rho}(\epsilon):=\int_{\rho}e^{-\eta/\epsilon}\mathcal{B}[\dot{g}](\eta)\mathrm{d\eta}\,.
\end{equation}
While \eqref{dotgintro} is frequently divergent for $\epsilon \neq 0$, if $\dot{g}_{\rho}(\epsilon)$ exists, it is an analytic function of $\epsilon$ in some sector of $\mathbb{C}$, whose asymptotic expansion as $\epsilon \to 0$ reproduces $\dot{g}$ (see \cite{MitschiSauzin2016DSSR1} for more information on resurgence). \\

Whenever $\dot{g}$ is resurgent one can associate the so-called Stokes automorphisms via the alien calculus. When the Borel sums $\dot{g}_{\rho}$ exist for rays $\rho$ avoiding $S$, the Stokes automorphisms can be interpreted as the jumps of $\dot{g}_{\rho}$ as we vary $\rho$ across points of $S$. In applications of resurgence to topological string theory and Riemann-Hilbert problems, the Stokes automorphisms have been associated to DT/BPS-invariants (see for example \cite{GPKM,AMM,  ASTT,BriTul}). Given the relation of Joyce structures to DT-invariants, it would be very interesting to study whether $\dot{g}$ is resurgent, and if so, its Stokes automorphisms and their relations to the aforementioned Riemann-Hilbert problem. \\

In Theorem \ref{mt2} we show that the $\dot{g}$ built in Theorem \ref{mt1} always has convergent $\mathcal{B}[\dot{g}]$, and give a necessary and sufficient condition for convergence of $\mathcal{B}[\dot{g}]$ for general $\dot{g}$ gauging $\mathcal{A}^{\text{st}}$ to $\mathcal{A}$. In Theorem \ref{mt3} we show a similar result for the formal twistor Darboux coordinates built from $\dot{g}$. \\

While the results of the previous paragraph are the first step towards showing that $\dot{g}$ and the formal twistor Darboux coordinates are resurgent,  we still do not have a solid understanding of the analytic continuations of their Borel transforms. We comment on a possible approach in Section \ref{endlesssec}, and intend to come back to this issue in future work. \\

We finish by presenting two examples associated to the DT-theory of the A1 and A2 quiver. While the A1 example is simple and everything can be done explicitly, we will show that wildly different things can happen for formal gauge transformations gauging $\mathcal{A}$ to $\mathcal{A}^{\text{st}}$. In particular, we will construct two such formal gauge transformations: one having a convergent $\dot{g}$ (and hence entire Borel transform) and one having a divergent $\dot{g}$. The one with divergent $\dot{g}$ turns out to be resurgent, and the analytic continuation of its Borel transform has the expected singularity structure and Stokes automorphisms related to the DT theory of the A1 quiver. Finally, in the case of the Joyce structure associated to the A2 quiver we compute the first two vector fields $\dot{g}_1$ and $\dot{g}_2$ in the $\epsilon$-expansion of a $\dot{g}$, in terms of explicit algebraic functions. These in turn can be used to compute expansions of Darboux twistor coordinates and associated objects like $\tau$-functions. 

\subsection{Structure of the paper}
\begin{itemize}
    \item In Section \ref{sec2} we introduce several preliminaries needed for the main results of the paper. This includes a quick summary on Ehresmann connections, gauge transformations, relative connections, and their formal analogues. We will not need the full generality of formal analytic spaces, so the discussion of formal analytic spaces and morphisms between them in Section \ref{formal_setting} will be only limited to our simple case. We finish with a summary of Joyce structures and their twistor spaces. 
    \item In Section \ref{sec3} we introduce the trivial Joyce structure associated to a holomorphic symplectic manifold with period structure, and prove  in Theorem \ref{mt1} that any Joyce structure can be formally gauged to the trivial one. In Proposition \ref{flatcoordsprop} we produce formal twistor Darboux coordinates by using the previous formal gauge transformation.  
    \item In Section \ref{sec4} we study the Borel transforms of infinitesimal gauge transformations and formal twistor Darboux coordinates. We show in Theorem \ref{mt2} that the Borel transform of the infinitesimal gauge transformation from Theorem \ref{mt1} converges in a neighbourhood of $\xi=0$, and in Proposition \ref{flatcoordsprop} we show an analogous result for the formal twistor Darboux coordinates. 
    \item In Section \ref{exsec} we illustrate applications of the previous results with two examples associated to the DT-theory of the A1 and A2 quiver.  
\end{itemize}

\subsection{Conventions}

Unless otherwise stated, all our objects and morphisms will be holomorphic. We will use the Einstein summation convention when writing tensors in coordinates. \\

\textbf{Acknowledgements:} the author is very grateful to T. Bridgeland for many helpful discussions.

\section{Preliminaries}\label{sec2}

In this section we will recall some basic notions about connections in the sense of Ehresmann (see for example \cite{ECref}), Joyce structures, and the formal setting that we will work with. Readers familiar with these notions should quickly skim through the section to be aware of the notations. \\

While many notions presented in this section hold in more generality, all our objects and morphisms will be holomorphic unless otherwise stated. 

\subsection{Ehresmann connections}
Let us fix a holomorphic surjective submersion $\pi:E\to B$. 
\begin{definition}
    An Ehresmann connection on $\pi:E\to B$ is a vector bundle map $\mathcal{A}:\pi^*(TB)\to TE$ such that $\pi_*\circ \mathcal{A}=\text{Id}_{\pi^*(TB)}$, where $\pi_*:TE\to \pi^*(TB)$ is the pushforward induced $\pi$,  and $\text{Id}_{\pi^*(TB)}$ is the identity map of $\pi^*(TB)$. 
\end{definition}
\begin{remark}\label{remconnot} In the following we will refer to Ehresmann connections simply as connections, and if $X\in \pi^*(TB)$, we will frequently denote the evaluation of $\mathcal{A}$ on $X$ by $\mathcal{A}_X:=\mathcal{A}(X)$. Given any local vector field $X$ on  $B$, we have the canonically induced pullback section $\pi^*X$ of $\pi^*(TB)\to E$. When we evaluate a connection $\mathcal{A}$ on a pullback section, we will abuse notation and simply denote it by $\mathcal{A}_{X}$ instead of $\mathcal{A}_{\pi^*X}$.\end{remark}

Given $\pi:E\to B$ as before, we have the associated vertical vector bundle $V_{\pi}\to E$ where the fiber over $e\in E$ is given by $V_{\pi}|_{e}:=\text{Ker}(\pi_*|_e)\subset T_eE$. If we are given any connection $\mathcal{A}$ on $\pi:E\to B$, we have a splitting
\begin{equation*}
    TE=\text{Im}(\mathcal{A})\oplus V_{\pi}
\end{equation*}
where $\text{Im}(\mathcal{A})$ denotes the image of $\mathcal{A}$. In particular, $\text{Im}(\mathcal{A})$ is horizontal with respect to $\pi$, and $\mathcal{A}$ defines a horizontal lift of tangent vectors of $B$. 

\begin{definition}\label{flatdef}
    A connection $\mathcal{A}$ on $\pi:E\to B$ is flat if the horizontal distribution $\text{Im}(\mathcal{A})\subset TE$ is involutive. That is 
    \begin{equation*}
        [\text{Im}(\mathcal{A}), \text{Im}(\mathcal{A})]\subset \text{Im}(\mathcal{A})\,.
    \end{equation*}
\end{definition}
Given a connection $\mathcal{A}$ on $\pi:E\to B$, a smooth path $\gamma:[0,1]\to B$, and a point $e\in E_{\gamma(0)}:=\pi^{-1}(\gamma(0))$, there exists a unique path $\alpha:[0,\delta]\to E$ for some $\delta>0$ such that $\alpha(0)=e$ and 
\begin{equation*}
    \frac{\mathrm{d}}{\mathrm{d}t}\alpha(t)= \mathcal{A}\left(\frac{\mathrm{d}}{\mathrm{d}t}\gamma(t)\right)\,.
\end{equation*}
\begin{definition}
    The path $\alpha$ is called the horizontal lift of $\gamma$ with respect to $\mathcal{A}$, starting at $e$. For $t\in [0,\delta]$, the parallel transport of $e$ along $\gamma$ for time $t$ is given by $\alpha(t)$.
\end{definition}
Given a fixed $\delta>0$ and $t<\delta$, we can assemble all horizontal lifts of $\gamma$ defined on $[0,t]$ in a single diffeomorphism between open subsets of the fibers $E_{\gamma(0)}$ and $E_{\gamma(t)}$. More specifically, there are open subsets $U_{0}\subset E_{\gamma(0)}$ and $U_t\subset E_{\gamma(t)}$ and a diffeomorphism $\text{PT}_t:U_0\to U_t$ such that $\text{PT}_t(e)=\alpha(t),$ where $\alpha$ is the horizontal lift of $\gamma$ with $\alpha(0)=e.$

\subsubsection{Gauge transformations}
We now recall the notion of gauge transformations in the context of submersions, and their action on connections. 

\begin{definition}A gauge transformation on a surjective holomorphic submersion $\pi:E\to B$ is a biholomorphism $g:E\to E$ such that $\pi \circ g=\pi$.
\end{definition}
Gauge transformations of $\pi:E\to B$ form a group $\mathcal{G}$ and act on the right on connections $\mathcal{A}$ of $\pi$ as follows. Given $(e,X)\in \pi^*(TB)$, we define $ (\mathcal{A}\cdot g):\pi^*(TB)\to TE$ as 
\begin{equation}\label{gaugetrans}
    (\mathcal{A}\cdot g)(e,X)=\mathrm{d}g^{-1}|_{g(e)}(\mathcal{A}(g(e),X))\,.
\end{equation}
The map $\mathcal{A}\cdot g$ satisfies the following with respect to $\pi_*$:
\begin{equation*}
    \pi_*\circ (\mathcal{A}\cdot g)(e,X)=(e, \mathrm{d}\pi(\mathrm{d}g^{-1}|_{g(e)}(\mathcal{A}(g(e),X)))=(e, \mathrm{d}\pi|_{g(e)}(\mathcal{A}(g(e),X)))=(e,X)\,,
\end{equation*}
so $\mathcal{A}\cdot g$ is a connection. Furthermore, if we have a path of gauge transformations $g_t$ with $t\in (-\epsilon,\epsilon)$ and such that $g_{0}=\text{Id}_{E}$, we have the associated infinitesimal gauge transformation $\dot{g}$ defined by
\begin{equation*}
    \dot{g}_e:=\frac{\mathrm{d}g_t(e)}{\mathrm{d}t}\Bigg|_{t=0}\in T_eE\,.
\end{equation*}
It is a vector field on $E$ that satisfies $\mathrm{d}\pi(\dot{g})=0$ (since $\pi\circ g_t=\pi$), so $\dot{g}$ is a vertical vector field on $E$. Conversely, if we start with a vertical vector field $\dot{g}$ on $E$ and assume that the flow $\phi_t$ exists for a time interval that is uniform for all points of $E$, then each $\phi_t:E\to E$ is a gauge transformation. 
\subsubsection{Relative connections}

Consider two surjective holomorphic submersions $\pi:E\to B$ and $q:B\to S$. Furthermore, denote by $V_{q}\to B$, $V_{\pi}\to E$, and  $V_{q\circ \pi}\to E$, the vertical bundles associated to $q$, $\pi$, and $q\circ \pi$, respectively. 

\begin{definition} A relative  connection on $\pi:E\to B$, relative to $q:B\to S$, is a bundle map 
\begin{equation*}
    \mathcal{A}:\pi^*V_q\to V_{q\circ \pi}
\end{equation*}
such that $\pi_*\circ \mathcal{A}=\text{Id}_{\pi^*V_q}$.\end{definition}
This means that we have the splitting
\begin{equation*}
    V_{q\circ \pi}=\text{Im}(\mathcal{A})\oplus V_{\pi}.
\end{equation*}
In other words, $\mathcal{A}$ defines a horizontal lift of tangent vectors of $B$ that are vertical with respect to $q:B\to S$, and the lifts are vertical with respect to $q\circ \pi:E\to S$. Note that the previous group $\mathcal{G}$ of gauge transformations of $\pi:E\to B$ acts on the right on relative connections with the same formula \eqref{gaugetrans} as before.  In the same way as with flatness of connections (Definition \ref{flatdef}), we will say that the relative connection $\mathcal{A}$ is flat if $\text{Im}(\mathcal{A})\subset V_{q\circ \pi}\subset TE$ is involutive.\\

We will only consider relative connections in the following simple setting. We take
\begin{equation*}
    E = TM\times \mathbb{C}, \quad B=M\times \mathbb{C}, \quad S= \mathbb{C},
\end{equation*}
and if $\pi:TM\to M$ is the canonical projection of the tangent bundle, we consider the projections
\begin{equation*}
\begin{split}
    p&:E\to B, \quad p(X,\lambda)=(\pi(X),\lambda),\\
    q&:B \to S, \quad q(b,\lambda)=\lambda.
\end{split}
\end{equation*}
Restricting the relative connection $\mathcal{A}$ to $TM\times \{\epsilon\}$ we obtain a bundle map 
\begin{equation*}
    \mathcal{A}^{\epsilon}:\pi^*V_q|_{TM\times \{\epsilon\}}\to V_{q\circ \pi}|_{TM\times \{\epsilon\}}\,.
\end{equation*}
Noting that $p^*V_q|_{TM\times \{\epsilon\}}\cong \pi^*(TM)$ and $V_{q\circ \pi}|_{TM\times \{\epsilon\}}\cong T(TM)$, we see that we obtain a usual connection on $\pi: TM\to M$. Hence, the data of a relative connection in this case can be thought as a family of connections on $TM\to M$ parametrized by $\mathbb{C}$.

\begin{remark}\label{flatrem}
    Note that the flatness of  relative connection $\mathcal{A}$ on $p:TM\times \mathbb{C}\to M\times \mathbb{C}$ is equivalent to the condition that for any two local vector fields $X,Y$ on $M$, we have 
    \begin{equation}\label{flatsimple}
        [\mathcal{A}_X,\mathcal{A}_Y]=\mathcal{A}_{[X,Y]}\,.
    \end{equation}
    In particular, if $(Z^i)$ are local coordinates on $M$, it is enough to check that 
    \begin{equation}\label{flatcoords}
        [\mathcal{A}_{\partial_{Z^i}}, \mathcal{A}_{\partial_{Z^j}}]=0\,, \quad i,j=1,...,\text{dim}_{\mathbb{C}}(M)\,.    \end{equation}
\end{remark}
\subsection{Formal completions and formal relative connections}\label{formal_setting}
In the previous section we considered a relative connection $\mathcal{A}$ on $p:TM\times \mathbb{C}\to M\times \mathbb{C}$, relative to the projection $q:M\times \mathbb{C}\to \mathbb{C}$. In what follows, we will need to  replace the $\mathbb{C}$-factors with the formal disk $\widehat{D}=\text{Spf}(\mathbb{C}[[\epsilon]])$, where the latter denotes the formal spectrum of $\mathbb{C}[[\epsilon]]$. To do this we will need to use some simple notions from the theory of formal analytic spaces \cite{GrauertRemmert1984}. \\

Given the complex manifold $M\times \mathbb{C}$, we have an associated analytic space given by considering the locally ringed space $(M\times \mathbb{C},\mathcal{O}_{M\times \mathbb{C}})$ where $\mathcal{O}_{M\times \mathbb{C}}$ is the sheaf of holomorphic functions on $M\times \mathbb{C}$. We now define $M\times \widehat{D}$ as follows

\begin{definition} $M\times \widehat{D}$ is the formal analytic space \begin{equation*}M\times \widehat{D}:=(M,\widehat{\mathcal{O}}_{M})\end{equation*} obtained by doing the formal completion of $M\times \mathbb{C}$\ along $M\times \{0\}$. Namely, given the canonical coordinate $\epsilon$ on the $\mathbb{C}$-factor, we consider the ideal sheaf $\mathcal{I}=(\epsilon)\subset \mathcal{O}_{M\times \mathbb{C}}$ and take the inverse limit
\begin{equation*}
    \widehat{\mathcal{O}}_M:=\lim_{\longleftarrow}\frac{\mathcal{O}_{M\times \mathbb{C}}}{\mathcal{I}^{n}}\,.
\end{equation*}
The sheaf $\widehat{\mathcal{O}}_{M}$ over $M$ satisfies
\begin{equation*}
    \widehat{\mathcal{O}}_{M}(U)=\mathcal{O}_M(U)\otimes \mathbb{C}[[\epsilon]]\,, \quad U\subset M\,.
\end{equation*}
The same applies for going from $TM\times \mathbb{C}$ to $TM\times \widehat{D}=(TM,\widehat{\mathcal{O}}_{TM})$, while $\widehat{D}$ is simply the formal analytic space $\widehat{D}=(\{0\}, \mathbb{C}[[\epsilon]])$.
\end{definition}

The maps $\widehat{p}:TM\times \widehat{D}\to M\times \widehat{D}$ and $\widehat{q}:M\times \widehat{D}\to \widehat{D}$ are now morphisms between formal analytic spaces corresponding to $p$ and $q$.  Morphisms $\widehat{f}=(f,f^{\sharp})=\widehat{X}\to \widehat{Y}$ between formal analytic spaces $\widehat{X}=(X,\widehat{\mathcal{O}}_{X})$ and $\widehat{Y}=(Y,\widehat{\mathcal{O}}_{Y})$ consist of a continuous map $f:X\to Y$ between the underlying topological spaces, together with a morphism of sheaves $f^{\sharp}:\mathcal{O}_{\widehat{Y}}\to f_*\mathcal{O}_{\widehat{X}}$ between topological local rings. In this context $\widehat{p}$ and $\widehat{q}$ are given in the following definition.

\begin{definition}
    The morphism $\widehat{p}:TM\times \widehat{D}\to M\times \widehat{D}$ of formal analytic spaces is given by $\widehat{p}=(\pi,p^{\sharp})$ where $\pi:TM\to M$ is the canonical projection and 
\begin{equation}\label{psharp}
    p^{\sharp}:\widehat{\mathcal{O}}_{M} \to \pi_*\widehat{\mathcal{O}}_{TM}, \quad p^{\sharp}\left(\sum_{k=0}^{\infty}f_k\epsilon^k\right)=\sum_{k=0}^{\infty}\pi^*(f_k)\epsilon^k\,.
\end{equation}
Similarly, $\widehat{q}$ is given by the projection $q|_{M\times \{0\}}:M\to \{0\}$ and $q^{\sharp}$ is given by pullback by $q$ as in \eqref{psharp}.  
\end{definition}

We now define the formal analogue of a relative connection $\mathcal{A}$ on $TM\times \mathbb{C}\to M\times \mathbb{C}$, relative to $q:M\times \mathbb{C}\to \mathbb{C}$.

\begin{definition}
A relative connection $\mathcal{A}$ on $\widehat{p}:TM\times \widehat{D}\to M\times \widehat{D}$ relative to the canonical projection $\hat{q}:M\times \widehat{D}\to \widehat{D}$ is given by a $\mathbb{C}[[\epsilon]]$-linear bundle map 
\begin{equation*}
    \mathcal{A}:\pi^*(TM)\otimes \mathbb{C}[[\epsilon]]\to T(TM)\otimes \mathbb{C}[[\epsilon]]\,
\end{equation*}
such that $\pi_*\circ \mathcal{A}=\text{Id}_{\pi^*(TM)\otimes \mathbb{C}[[\epsilon]]}$.

\end{definition}
One can check that a formal relative connection $\mathcal{A}$ as above is given by a formal expression of the form
\begin{equation*}
    \mathcal{A}=h+\sum_{k=1}^{\infty}\omega^{(k)}\epsilon^k\,,
\end{equation*}
where $h$ and $\omega^{(k)}$ are $\mathbb{C}[[\epsilon]]$-linear extensions of a connection $h:\pi^*(TM)\to T(TM)$ on $\pi: TM\to M$, and bundle maps $\omega^{(k)}:\pi^*(TM)\to V_{\pi}$. $\mathcal{A}$ is indeed relative to $\hat{q}:M\times \widehat{D}\to \widehat{D}$ since it does not lift any vectors in the formal direction $\partial_{\epsilon}$, and all lifted vectors point along $TM$.\\

As in Remark \ref{flatrem}, the flatness of the relative connection $\mathcal{A}$ on $\hat{p}$ is equivalent to checking \eqref{flatsimple} or \eqref{flatcoords}.
\subsubsection{Formal gauge transformations}
We would like to consider formal gauge transformations of $\widehat{p}:TM\times \widehat{D}\to M\times \widehat{D}$. We will first  describe the infinitesimal gauge transformations vanishing at $\epsilon=0$. 

\begin{definition}\label{formalgauge} 
The Lie algebra of infinitesimal gauge transformations of $\widehat{p}:TM\times \widehat{D}\to M\times \widehat{D}$ vanishing at $\epsilon =0$ is given by 
\begin{equation*}
    \text{Lie}(\widehat{\mathcal{G}}):=\epsilon\cdot \Gamma(TM,V_{\pi})[[\epsilon]]=\left\{\dot{g}=\sum_{k=1}^{\infty}\dot{g}_k\epsilon^k \quad \Bigg| \quad \dot{g}_k\in \Gamma(TM,V_{\pi})\right\},
\end{equation*}
where $\Gamma(TM,V_{\pi})$ denotes  sections of $V_{\pi}\to TM$ (i.e. vector fields on $TM$, vertical with respect to $\pi:TM \to M$).
\end{definition}
Given $\dot{g}\in \text{Lie}(\widehat{\mathcal{G}}),$ the exponential $e^{\dot{g}}$ generates a gauge transformation of $\widehat{p}:TM\times \widehat{D}\to M\times \widehat{D}$ extending the identity map of $TM\times \{0\}$. Namely, we have the morphism of formal analytic spaces $\hat{g} = (\text{Id}_{TM}, e^{\dot{g}}):TM\times \widehat{D}\to TM\times \widehat{D}$ where $e^{\dot{g}}: \widehat{\mathcal{O}}_{TM}\to \widehat{\mathcal{O}}_{TM}$ acts on functions $f\in \hat{\mathcal{O}}_{TM}(U)\cong \mathcal{O}_{TM}(U)[[\epsilon]]$ by \begin{equation}\label{formgauge}
    e^{\dot{g}}f=\sum_{n=0}^{\infty}\frac{\dot{g}^k}{k!}f\,.
\end{equation}
Note that $\hat{g}$ satisfies $\hat{\pi}\circ \hat{g}=\hat{\pi}$. 
Indeed, for functions $h\in \mathcal{O}_{TM}(\pi^{-1}(U))[[\epsilon]]$ of the form $h=\sum_{k=0}^{\infty}(\pi^*h_k)\epsilon^k$ we have $e^{\dot{g}}h=h,$ so it fixes all functions on $TM\times \widehat{D}$ coming from the base.
$e^{\dot{g}}$ should be thought at the formal analog of the action of the time $1$ flow of $\dot{g}$ on the function $f$ in the setting of usual submersions between complex manifolds. 

\begin{remark} Note that the fact that $\dot{g}$ vanishes at $\epsilon=0$ guarantees that \eqref{formgauge} makes sense as a formal series in $\epsilon$. In other words, the coefficient of each power $\epsilon^k$ gets finitely many contributions when expanding \eqref{formgauge} in powers of $\epsilon$. 
\end{remark}

\begin{definition}
    We denote the formal gauge transformations of $\widehat{p}:TM\times \widehat{D}\to M\times \widehat{D}$ extending the identity at $\epsilon=0$ by
    \begin{equation*}
        \widehat{\mathcal{G}}=\{\widehat{g}=(\text{Id}_{TM}, e^{\dot{g}}), \quad \dot{g}\in \text{Lie}(\widehat{\mathcal{G}})\}\,.
    \end{equation*}
    The set $\widehat{\mathcal{G}}$ is a group via the Baker-Campbell-Hausdorff (BCH) formula. 
\end{definition}

On the other hand, $\dot{g}$ also generates an action on (relative)  connections on $\hat{p}$ via the exponential of the Lie derivative $\mathcal{L}_{\dot{g}}=\sum_{k=1}^{\infty}\mathcal{L}_{\dot{g}_k}\epsilon^k$. Indeed,  given $X=(V,W)\in \pi^*(TM)$, we can extend $W$ to a local vector field of $M$, and consider the induced pullback section $\widetilde{X}$ of $\pi^*(TM)$ satisfying $\widetilde{X}|_V=X$. We then define $(e^{\mathcal{L}_{\dot{g}}}\mathcal{A})(X)$ by 
\begin{equation}\label{formgaugeconn}
    (e^{\mathcal{L}_{\dot{g}}}\mathcal{A})(X):= \sum_{k=0}^{\infty}\frac{\mathcal{L}_{\dot{g}}^k(\mathcal{A}_{\widetilde{X}})}{k!}\,\Bigg|_{V}\in T_V(TM)\otimes \mathbb{C}[[\epsilon]]\,.
\end{equation}
Note that the sum in \eqref{formgaugeconn} makes sense as a formal power series in $\epsilon$ since $\dot{g}$ does not have an $\epsilon^{0}$-term. Furthermore, it is easy to check that the right hand side of \eqref{formgaugeconn} does not depend of the extension of $W$, and that $\mathcal{L}_{\dot{g}}^k(\mathcal{A}_{\widetilde{X}})$ is always vertical for $k\geq 1$. Extending \eqref{formgaugeconn} linearly in $\mathbb{C}[[\epsilon]]$ we obtain a $\mathbb{C}[[\epsilon]]$-linear bundle map 

\begin{equation*}
    e^{\mathcal{L}_{\dot{g}}}\mathcal{A}:\pi^*(TM)\otimes \mathbb{C}[[\epsilon]]\to T(TM)\otimes \mathbb{C}[[\epsilon]]\,,
\end{equation*}
satisfying $\pi_*\circ \mathcal{A}=\text{Id}_{\pi^*(TM)\otimes \mathbb{C}[[\epsilon]]}$, and hence a connection on $\hat{p}$, relative to $\hat{q}$.\\

The action of $e^{\mathcal{L}_{\dot{g}}}$ on $\mathcal{A}$ is the formal analog of the action by gauge transformations on  $\mathcal{A}$ given by the time $1$ flow of infinitesimal gauge transformation $\dot{g}$ in the setting of usual submersions between complex manifolds. \\

For completeness we end this section by proving the following proposition, which can be skipped on a first reading. While it is clear that gauge transformations preserve flatness of connections in the usual setting of gauge transformations on surjective submersions, the fact that the same holds for the action of formal gauge transformation might not be immediately clear from \eqref{formgaugeconn}. 
\begin{lemma}\label{flatgaugelemma}
    If a relative connection (relative to $\hat{q}$) $\mathcal{A}$ on $\widehat{p}:TM\times \widehat{D}\to M\times \widehat{D}$ is flat, then $e^{\mathcal{L}_{\dot{g}}}\mathcal{A}$ is flat. 
\end{lemma}
\begin{proof}
     Choose local coordinates $(Z^i)$ on $M$, and 
    let $\mathcal{A}_i:=\mathcal{A}_{\partial_{Z^i}}$. To show that $e^{\mathcal{L}_{\dot{g}}}\mathcal{A}$ is flat, it is enough to show that for arbitrary $i$ and $j$ \begin{equation}\label{gaugeformalflat}
        [e^{\mathcal{L}_{\dot{g}}}\mathcal{A}_i,e^{\mathcal{L}_{\dot{g}}}\mathcal{A}_j]=0\,.
    \end{equation}
    If we denote by $B^{k}(\mathcal{A}_i):=\mathcal{L}_{\dot{g}}^k(\mathcal{A}_i)$, then we can write the above condition as
    \begin{equation*}
        \sum_{k,l=0}^{\infty}\frac{1}{k!l!}[B^k(\mathcal{A}_i),B^l(\mathcal{A}_j)]=0\,.
    \end{equation*}
    We have 
    \begin{equation*}
        \sum_{k,l=0}^{\infty}\frac{1}{k!l!}[B^k(\mathcal{A}_i),B^l(\mathcal{A}_j)]=\sum_{m=0}^{\infty}\frac{1}{m!}\sum_{k+l=m}\binom{m}{k}[B^k(\mathcal{A}_i),B^l(\mathcal{A}_j)]=\sum_{m=0}^{\infty}\frac{1}{m!}\sum_{k=0}^{m}\binom{m}{k}[B^k(\mathcal{A}_i),B^{m-k}(\mathcal{A}_j)]\,.
    \end{equation*}
    On the other hand, by applying Jacobi's identity we can show using induction that 
    \begin{equation*}
        [B^k(\mathcal{A}_i),B^{m-k}(\mathcal{A}_j)]=\sum_{l=0}^{k}\binom{k}{l}(-1)^lB^{k-l}([\mathcal{A}_i,B^{m-k+l}(\mathcal{A}_j)])\,.
    \end{equation*}
    We will show that for each $m$ 
    \begin{equation*}
        \sum_{k=0}^{m}\binom{m}{k}\sum_{l=0}^{k}\binom{k}{l}(-1)^lB^{k-l}([\mathcal{A}_i,B^{m-k+l}(\mathcal{A}_j)])=0\,.
    \end{equation*}
    Setting $p=k-l$ we have 
    \begin{equation}\label{sumreorg1}
        \begin{split}
        \sum_{k=0}^{m}&\binom{m}{k}\sum_{l=0}^{k}\binom{k}{l}(-1)^lB^{k-l}([\mathcal{A}_i,B^{m-k+l}(\mathcal{A}_j)])=\sum_{p=0}^{m}\sum_{l=0}^{m-p}\binom{m}{l+p}\binom{l+p}{l}(-1)^lB^p([\mathcal{A}_i,B^{m-p}(\mathcal{A}_j)])\\
        &=\sum_{p=0}^{m}B^p([\mathcal{A}_i,B^{m-p}(\mathcal{A}_j)])\sum_{l=0}^{m-p}\binom{m}{l+p}\binom{l+p}{l}(-1)^l=\sum_{p=0}^{m}B^p([\mathcal{A}_i,B^{m-p}(\mathcal{A}_j)])\sum_{l=0}^{m-p}\binom{m}{p}\binom{m-p}{l}(-1)^l\\
        &=\sum_{p=0}^{m}B^p([\mathcal{A}_i,B^{m-p}(\mathcal{A}_j)])\binom{m}{p}\sum_{l=0}^{m-p}\binom{m-p}{l}(-1)^l.\\
        \end{split}
    \end{equation}
    Using that \begin{equation}
        (a+b)^n=\sum_{k=0}^n\binom{n}{k}a^kb^{n-k}
    \end{equation}
    we see that \eqref{sumreorg1} is clearly $0$ if $m\neq p$, while in the case $m= p$ we get $0$ due to the flatness of $\mathcal{A}$. Hence, it follows that \eqref{gaugeformalflat} holds.
\end{proof}
\subsection{Joyce Structures}\label{Joyce_section}

We now briefly recall the notion of Joyce Structure, following \cite{TwistorJoyce}. Let $(M,\Omega)$ be a holomorphic symplectic manifold, and $\pi:TM\to M$ the canonical projection. $\Omega$ induces a fiber-wise symplectic structure $\Omega^v$ on the the vertical bundle $V_{\pi}\to TM$ via the canonical identification $v:\pi^*TM\to V_{\pi}$. More specifically, if $p_2:\pi^*TM\to TM$ denotes the projection on the second factor, we have
\begin{equation}\label{verthol}
    \Omega^v:=(p_2\circ v^{-1})^*\Omega\,.
\end{equation}

\begin{definition} If we are given a non-linear connection $h$ on $\pi:TM \to M$, we will say that $h$ is symplectic with respect to $(M,\Omega)$ if $\Omega^{v}$ is invariant under parallel transport. Namely, for any path $\gamma$ on $M$, the corresponding parallel transport map $\text{PT}_{t}:U_{0}\to U_t$ between open subsets $U_0\subset E_{\gamma(0)}$ and $U_t\subset E_{\gamma(t)}$ satisfies
\begin{equation*}
    \text{PT}_t^*(\Omega^v)=  \Omega^v\,.
\end{equation*}
Note that this is well-defined, since $\mathrm{d}(\text{PT}_{t})$ sends vertical vectors to vertical vectors. 
\end{definition}

In \cite[Lemma 2.3]{TwistorJoyce} some simple criteria are stated to check that a connection is symplectic without having to work with parallel transports. 

\begin{definition} A pre-Joyce structure on a holomorphic symplectic manifold $(M,\Omega)$ is a connection $h$ on the canonical projection $\pi:TM\to M$ such that for each $\epsilon\in \mathbb{C}-\{0\}$, the connection \begin{equation}\label{EC}
    \mathcal{A}^{\epsilon}:=h+\epsilon^{-1}v
\end{equation}
is flat and symplectic. 
\end{definition}
\begin{remark}\label{merrem}
    \begin{itemize}
        \item In order to include some interesting examples, one should weaken the above definition and allow $h$ to be meromorphic rather than holomorphic \cite[Remark 2.5]{TwistorJoyce}. In particular, for some effective divisor $D\subset TM$ we should have a bundle map $h:\pi^*(TM)\to T(TM)\otimes \mathcal{O}_{TM}(D)$ satisfying 
        \begin{equation*}
            (\pi_*\otimes \mathcal{O}_{TM}(D))\circ h=\text{Id}_{\pi^*(TM)}\otimes s_D
        \end{equation*}
        where $s_D:\mathcal{O}_{TM}\to \mathcal{O}_{TM}(D)$ is the canonical inclusion. While everything below extends to this more general case, we will write everything with $h$ holomorphic in order to lighten the notation. 
        \item Note that the family of connections $\mathcal{A}^{\epsilon}$ on $\pi:TM\to M$ induces a relative connection on $TM\times \mathbb{C}\to M\times \mathbb{C}$, with a simple pole at $\epsilon=0$ with residue $v$. In particular, it also induces a connection on the formal bundle $\hat{\pi}:TM\times \widehat{D}\to M\times \widehat{D}$ with the same pole structure, where $\widehat{D}$ is the formal disk. \\
    \end{itemize}
\end{remark}

To have a full-Joyce structure we need additional symmetries and structures on top of a pre-Joyce structure, which we now describe.

\begin{definition} A period structure on $M$ is a tuple $(\Gamma, Z)$ where:
\begin{itemize}
    \item $\Gamma\to M$ is a bundle of full lattices of $TM\to M$. In other words, for each $p\in M$, $\Gamma_p\subset T_pM$ is a discrete, finitely generated abelian group such that $\Gamma_p\otimes_{\mathbb{Z}}\mathbb{C}= T_pM$.
    \item If $\nabla$ is the flat linear connection on $M$ determined by $\Gamma$, then $Z$ is a holomorphic vector field on $M$ such that $\nabla (Z) = \text{Id}_{TM}$.
\end{itemize}
\end{definition}
Note that with respect to a local $\nabla$-flat frame of $TM\to M$, the second condition implies that the components of $Z$ in that frame induce local flat coordinates on $M$, and hence that $\nabla$ is torsion-free. \\

We now have all the ingredients to define a Joyce structure.

\begin{definition}\label{Joyce_def} A Joyce structure on a complex manifold $M$ is a pre-Joyce structure $(\Omega,h)$ on $M$ together with a period structure $(\Gamma,Z)$ such that:
\begin{itemize}
    \item (J1)  $(2\pi \mathrm{i})^{-1}\Omega^{-1}$ takes integral values on the dual bundle of lattices $\Gamma^*\subset T^*M$, where $\Omega^{-1}$ is the dual symplectic form.
    \item (J2) $h$ is invariant under translations by $(2\pi \mathrm{i})\cdot \Gamma \subset TM$.
    \item (J3) If $\nabla$ is the flat connection on $M$ induced by $\Gamma$ and $E$ denotes the $\nabla$-horizontal lift of $Z$, then for any vector field $X$ on $M$ \begin{equation}
        h([Z,X])=[E,h(X)]\,.
    \end{equation}
    \item (J4) If $\iota:TM \to TM$ denotes the involution that acts by $-1$ o the fibers of $\pi:TM\to M$, then $h$ is invariant under the action of $\iota$.
\end{itemize}
\end{definition}

\begin{remark}  \label{JoyceDefRem}
\begin{itemize}
    \item As previously mentioned in the introduction, the notion of Joyce structure is motivated by DT theory. In that context $M$ plays the role of the space of stability conditions of a 3d Calabi-Yau category, and carries a natural period structure $(\Gamma,Z)$ determined by the charge lattice and central charge. Furthermore, the holomorphic symplectic form $\Omega$ of the Joyce structure corresponds to the inverse of the Euler form of $\Gamma$. The additional structure of $h$ satisfying the properties (J1-4) is the conjectural geometric structure encoded by the DT invariants of the 3d Calabi-Yau category. 
    \item Note that condition (J1) implies that the flat torsion-free connection $\nabla$ also satisfies $\nabla \Omega =0$, and hence that $(M,\Omega,\nabla)$ is a flat holomorphic symplectic manifold. Condition (J2) implies that $h$ (and also $\mathcal{A}^{\epsilon}$) descends to a connection on the $(\mathbb{C}^*)^n$-bundle 
    \begin{equation*}
        \pi:TM/(2\pi \mathrm{i})\cdot \Gamma \to M\,.
    \end{equation*}
    Finally, conditions (J3) and (J4) are reflections of properties satisfied by the DT invariants. 
\end{itemize}

\end{remark}

We can use the previous extra structure to refine the form of $h$, as we now explain. We denote by $\mathcal{H}$ the Ehresmann connection induced by $\nabla$ on $\pi:TM \to M$. We can then write
\begin{equation}
    h_X=\mathcal{H}_X+\omega_X
\end{equation}
where $\omega:\pi^*(TM)\to V_{\pi}$. On the other hand, in terms of vector fields on $M$, the flatness condition for $\mathcal{A}^{\epsilon}$ can be written as
\begin{equation}\label{flatLie}
    [\mathcal{A}^{\epsilon}_X,\mathcal{A}^{\epsilon}_{Y}]=\mathcal{A}^{\epsilon}_{[X,Y]}.
\end{equation}
Using \eqref{flatLie} and the fact that $\mathcal{A}^{\epsilon}$ is symplectic,  it can be shown (see \cite{TwistorJoyce}) that we can locally find a holomorphic function $W: TM \to \mathbb{C}$ such that for vector fields $X$ on $M$
\begin{equation}\label{plebform}
    h_X = \mathcal{H}_X+\text{Ham}^{\Omega^{v}}({v_XW})\,.
\end{equation}
Here $v_XW$ denotes the action of $v_X$ on $W$, and $\text{Ham}^{\Omega^{v}}$ denotes the Hamiltonian vector field with respect to $\Omega^{v}$. The function $W$ is called the Pleba\'nski function and it is unique if we impose that on the zero section $M\subset TM$ and vector fields $X$ on $M$ \begin{equation}\label{plebzerosec}
    W|_M=v_XW|_M=0\,.
\end{equation} 
The flatness condition \eqref{flatLie} in terms of $W$ becomes a set of PDE's for $W$, called the Pleba\'nski second heavenly equations. Finally, $W$ must satisfy several symmetries corresponding to (J2)-(J4) (see [Lemma 3.7]\cite{TwistorJoyce}). 

\begin{remark}\label{extremark}
Consider the relative meromorphic connection $\mathcal{A}$ on $p:TM\times \mathbb{C}\to M\times \mathbb{C}$ associated to a Joyce structure. There is a natural way to extend $\mathcal{A}$ from a relative flat connection on $p$ to a full flat meromorphic connection by defining 
\begin{equation}\label{epsilonext}
    \mathcal{A}_{\epsilon \partial_{\epsilon}}=\epsilon\partial_{\epsilon}-\frac{1}{\epsilon}v_Z - \omega_Z\,,
\end{equation}
where $Z$ is the vector field from the period structure. The extension in turn induces a meromorphic flat connection on $\hat{p}:TM\times \widehat{D}\to M\times \widehat{D}$, which is in the class of almost standard connections from \cite[Section 5.2]{KSresurgence}, up to some different sign conventions. In \cite[Section 5.2]{KSresurgence}, almost standard connections are allowed to have a more general form than Joyce structures. Namely, the can have expansions of the form
\begin{equation*}
\begin{split}
    \mathcal{A}_X&=\mathcal{H}_X+\frac{1}{\epsilon}v_X+\sum_{k=0}^{\infty}\epsilon^k\omega_X^{(k)}, \quad X\in \pi^*(TM)\\
    \mathcal{A}_{\epsilon\partial_{\epsilon}}&=\epsilon\partial_{\epsilon}-\frac{1}{\epsilon}v_Z + \sum_{k=0}^{\infty}\epsilon^kv^{(k)}\,,
\end{split}
\end{equation*}
where $\omega_X^{(k)}$ and $v^{(k)}$ are vertical with respect to $\pi:TM\to M$.    
\end{remark}

\subsubsection{Twistor spaces of Joyce structures}\label{twistorJoycesec}

We end the preliminaries with a quick description of the twistor space associated to a Joyce structure, following \cite{TwistorJoyce}. Joyce structures over $M$ encode a complex hyperk\"{a}hler ($\mathbb{C}$-HK) geometry on $TM$ \cite{BriStr},\footnote{Or more precisely, an open subset of $TM$ if $h$ is meromorphic.} and the latter can be encoded in a twistor space $\mathcal{Z}$. A $\mathbb{C}$-HK structure is a holomorphic analog of a usual hyperk\"ahler (HK) structure, where the usual tensors of an HK structure are replaced by holomorphic tensors with respect to an underlying complex structure. More precisely:

\begin{definition}
    A $\mathbb{C}$-HK structure on a complex manifold $N$ is a tuple $(g,I,J,K)$ such that:
    \begin{itemize}
        \item $g$ is a non-degenerate holomorphic section of $\text{Sym}^2(T^*N)\to N$.
        \item $I$, $J$, $K$ are holomorphic sections of $\text{End}(TN)\to N$ satisfying the imaginary quaternion relations: $I^2=J^2=K^2=IJK=-1$.
        \item For $R\in \{I,J,K\}$ and $X,Y\in TN$ we have $g(RX,RY)=g(X,Y)$ and $\nabla(R)=0$, where $\nabla$ is the Levi-Civita connection associated to $g$.
    \end{itemize}
\end{definition}
Given $\mathbb{C}$-HK structure $(N,g,I,J,K)$, one has associated holomorphic symplectic forms on the complex manifold $N$
\begin{equation*}
    \Omega_R(-,-):=g(R-,-), \quad R\in \{I,J,K\}
\end{equation*}
which are the analogue of the K\"{a}hler forms of a usual HK structure. \\

In the case a Joyce structure $(M,\Omega,\Gamma,Z,h)$, the associated $\mathbb{C}$-HK structure on $N=TM$ is described in terms of $(\Omega,h)$. In particular, one has a $\mathbb{C}$-HK structure associated to any pre-Joyce structure, with the one corresponding to a Joyce structure enjoying some extra symmetries. The tensors $(I,J,K)$ are determined by 
\begin{equation*}
\begin{split}
    I\circ h &= i\cdot h, \quad J\circ h= -v, \quad K\circ h = i\cdot v\\
    I\circ v &= -i\cdot v, \quad J\circ v = h, \quad K\circ v = i\cdot h
\end{split}
\end{equation*}
while $g$ is determined by 
\begin{equation*}
g(h_X,v_Y)=\frac{1}{2}\Omega(X,Y), \quad g(h_X,h_Y)=g(v_X,v_Y)=0\,.
\end{equation*}

We now briefly explain how to construct the associated twistor space $\mathcal{Z}$ associated to a Joyce structure. Denoting $N=TM$ and 
fixing $ [\epsilon_0: \epsilon_1]\in \mathbb{P}^1$, consider the distribution on $N$ given by 
\begin{equation*}
H([\epsilon_0: \epsilon_1])=\text{Im}(\epsilon_0v+\epsilon_1h)\subset TN\,.
\end{equation*}
Denoting by $\pi_1:N\times \mathbb{P}^1\to N$ the projection into the first factor and varying $[\epsilon_0: \epsilon_1]$, we obtain a subbundle $H$ of $\pi_1^*(TN)\to N\times \mathbb{P}^1$, and by composing with the canonical inclusion $\pi_1^*(TN)\subset T(N\times \mathbb{P}^1)$ we obtain a distribution $H$ on $N\times \mathbb{P}^1$. The flatness $\mathcal{A}^{\epsilon}$ implies that this distribution is involutive, and hence generates a foliation. The twistor space $\mathcal{Z}$ is then the leaf space of this foliation, and hence comes with a quotient map $q:TM\times \mathbb{P}^1\to \mathcal{Z}$. Furthermore, the projection into the second factor $\pi_2:N\times \mathbb{P}^1\to \mathbb{P}^1$ induces a projection $p:\mathcal{Z}\to \mathbb{P}^1$ satisfying $p\circ q = \pi_2$. We will ignore some technicalities that may arise when defining $\mathcal{Z}$ as a leaf space, and assume that $\mathcal{Z}$ is a complex manifold (see \cite[Remark 4.1]{TwistorJoyce}). This will cause no issue, since all the statements that we will do about $\mathcal{Z}$ can be translated to precise statements about objects on $TM\times \mathbb{P}^1$. For example, a holomorphic function on $\mathcal{Z}$ will be described by a holomorphic function on $TM\times \mathbb{P}^1$ satisfying $\text{Ker}(\mathrm{d}f)=H$. \\

The twistor space $\mathcal{Z}$ comes with an extra structure that encodes the $\mathbb{C}$-HK structure associated to a Joyce structure. This extra structure is an $\mathcal{O}(2)$-twisted family of holomorphic symplectic forms relative to $p:\mathcal{Z}\to \mathbb{P}^1$. In other words, a holomorphic section 
\begin{equation}
    \varpi\in \Gamma(\mathcal{Z}, \Lambda^2V_{p}^*\otimes p^*(\mathcal{O}(2))),
\end{equation}
restricting to a holomorphic symplectic form on the fibers $\mathcal{Z}_{\epsilon}:=p^{-1}(\epsilon)$. 

\begin{definition}
    We call the tuple $(p:\mathcal{Z}\to \mathbb{P}^1, \varpi)$ the twistor space associated to the Joyce structure $(M,\Omega,\Gamma,Z,h)$. 
\end{definition}

\begin{remark}
    Contrary to the twistor space of a usual hyperk\"{a}hler manifold, the twistor space of a $\mathbb{C}$-HK structure does not come equipped with an antiholomorphic involution $\tau:\mathcal{Z}\to \mathcal{Z}$ covering the antipodal map of $\mathbb{P}^1$.
\end{remark}

Via the projection $q:N\times \mathbb{P}^1\to \mathcal{Z}$, it is easy to describe $q^*\varpi$ in terms of the three holomorphic symplectic forms $\Omega_I$, $\Omega_J$, and $\Omega_K$ associated to the $\mathbb{C}$-HK structure. Indeed, in terms of the affine coordinate $\epsilon = \epsilon_1/\epsilon_0$ on $\mathbb{P}^1$ we can write
\begin{equation}\label{twistholsymp}
    q^*\varpi=\left((\Omega_J+\mathrm{i}\Omega_K)+2\mathrm{i}\epsilon\Omega_I+\epsilon^2(\Omega_J-\mathrm{i}\Omega_K)\right)\otimes \partial_{\epsilon}\,.
\end{equation}
The $\mathcal{O}(2)$-twisted $2$-form $q^*\varpi$ is characterized by the conditions that for each fixed $\epsilon$ it is a closed 2-form satisfying 
\begin{equation*}
    \text{Ker}(q^*\varpi|_{TM\times \{\epsilon\}})=\text{Im}(\mathcal{A}^{\epsilon}) \quad \text{for} \quad \epsilon \in \mathbb{C}^{*}, \quad (q^*\varpi)(v_X,v_Y)=\Omega(X,Y)\epsilon^2\otimes \partial_{\epsilon}\,.
\end{equation*}

\section{Formal Classification of Joyce Structures}\label{sec3}

Consider a Joyce structure over $M$ with family of connections $\mathcal{A}^{\epsilon}$ on $\pi:TM \to M$. As we saw in the previous section, this induces a formal relative connection $\mathcal{A}$ on $\widehat{p}:TM \times \widehat{D}\to M\times \widehat{D}$, relative to $\widehat{q}:M\times \widehat{D}\to \widehat{D}$. Following some of the ideas in \cite[Section 5.3]{KSresurgence}, we show that $\mathcal{A}$ can be gauged to a standard form via the action of the formal gauge group  $\widehat{\mathcal{G}}$ introduced in Section \ref{formal_setting}. We then use this result to produce coordinates of the formal twistor space of the Joyce structure.   \\

Taking into account Remark \ref{merrem} about meromorphic Joyce structures, in order for some of the results in this and following sections to apply one must assume that the meromorphic Joyce structure has no poles along the zero section $M\subset TM$. This condition seems to hold in several cases of meromorphic Joyce structures \cite{BridgeFab,BridgelandOsc}.

\subsection{The trivial Joyce structure}\label{trivJoyceSec}

Suppose we are given a holomorphic symplectic manifold $(M,\Omega)$ with a period structure $(\Gamma,Z)$ such that the integrality condition (J1) from Definition \ref{Joyce_def} holds.  

\begin{proposition}\label{trivJoyceStructure} Given $(M,\Omega)$ and $(\Gamma,Z)$ as before, we can define a Joyce structure over $M$ by taking
\begin{equation*}
h=\mathcal{H},
\end{equation*}
where $\mathcal{H}$ is the connection on $\pi:TM\to M$ induced by the linear connection $\nabla$ coming from $\Gamma$.

\begin{proof}
    Since $(M,\Omega,\nabla)$ is a flat holomorphic symplectic manifold (recall Remark \ref{JoyceDefRem}), we can pick holomorphic Darboux coordinates $(Z^i)$ on $M$ which are flat with respect to $\nabla$. Consider the induced coordinates $(Z^i,\theta^i)$ on $TM$. With respect to these coordinates we have
    \begin{equation*}
        \mathcal{H}_{\partial_{Z^i}}=\partial_{Z^i}, \quad v_{\partial_{Z^i}}=\partial_{\theta^i}\,.
    \end{equation*}
    From this one easily checks that 
    \begin{equation*}
    [\mathcal{A}^{\epsilon}_{\partial_{Z^i}},\mathcal{A}^{\epsilon}_{\partial_{Z^j}}]=0,
    \end{equation*}
    from which flatness of $\mathcal{A}^{\epsilon}$ follows. On the other hand, if in the above coordinates we have
    \begin{equation*}
        \Omega=\Omega_{ij}\mathrm{d}Z^i\wedge \mathrm{d}Z^j, \quad \Omega_{ij}\in \mathbb{C}
    \end{equation*}
    then it is easy to check using \eqref{verthol} that 
    \begin{equation*}
    \Omega^v=\Omega_{ij}\mathrm{d}\theta^i\wedge \mathrm{d}\theta^j\,.
    \end{equation*}
    By \cite[Lemma 2.3]{TwistorJoyce}, to check that the flat connection $\mathcal{A}^{\epsilon}$ is symplectic it is enough to show that the unique extension $\Omega^{\epsilon} \in \Omega^2(TM)$ of $\Omega^{v}$ satisfying $\text{Ker}(\Omega^{\epsilon})=\text{Im}(\mathcal{A}^{\epsilon})$ is closed. In the previous coordinates $\Omega^{\epsilon}$ is explicitly given by\footnote{This is the 2-form factor of $\epsilon^{-2}q^*\varpi$, where $q^*\varpi$ is given in \eqref{twistholsymp}.}
    \begin{equation}\label{twisttriv}
    \Omega^{\epsilon}=\epsilon^{-2}\Omega_{ij}\mathrm{d}Z^i\wedge \mathrm{d}Z^j - 2\epsilon^{-1}\Omega_{ij}\mathrm{d}\theta^i\wedge \mathrm{d}Z^j + \Omega_{ij}\mathrm{d}\theta^i\wedge \mathrm{d}\theta^j\,.
    \end{equation}
    Since the coefficients $\Omega_{ij}$ are constant, it immediately follows that $\mathrm{d}\Omega^{\epsilon}=0$ and hence that $\mathcal{A}^{\epsilon}$ is symplectic. \\

    Finally, condition (J1) holds by assumption, while $(J2-4)$ follow trivially from the definition of $\mathcal{H}$. We then conclude that $(M,\Omega)$ and $(\Gamma,Z)$ together with $h=\mathcal{H}$ define a Joyce structure on $M$.    
\end{proof}

The Joyce structure constructed in Proposition \ref{trivJoyceStructure} can be thought as the trivial Joyce structure associated to $(M,\Omega,\Gamma,Z)$. Indeed, comparing to \eqref{plebform} we see that this Joyce structure corresponds to taking $W=0$ as a global Pleba\'nski function. Analogously, in terms of DT theory it can be thought as corresponding to the case where all the DT invariants are $0$.
\end{proposition}
\begin{definition}
    The Joyce structure from Proposition \ref{trivJoyceStructure} associated to a holomorphic symplectic manifold $(M,\Omega)$ with period structure $(\Gamma,Z)$ will be called the trivial Joyce structure on $M$.
\end{definition}
\begin{remark}
\begin{itemize}
\item There is an analogous notion in \cite[Section 5.1]{KSresurgence}, which they call the standard connection associated to trivial stability data. 
\item Note that the proof of Proposition \ref{trivJoyceStructure} shows more generally that given any flat holomorphic symplectic manifold $(M,\Omega,\nabla)$, we can define a pre-Joyce structure by taking $h=\mathcal{H}$, where $\mathcal{H}$ is the Ehresmann connection induced by $\nabla$.
\end{itemize}

\end{remark}
\subsection{Formally gauging to the trivial Joyce structure}

Suppose we are given a Joyce structure over $M$, corresponding to the data $(M,\Omega,\Gamma,Z,h)$. Using the notation from Section \ref{formal_setting}, the associated $\mathbb{C}^{*}$-family of flat symplectic connections \begin{equation*}
    \mathcal{A}^{\epsilon}=h+\frac{1}{\epsilon}v
\end{equation*} induces a meromorphic relative connection $\mathcal{A}$ on $p:TM\times \mathbb{C}\to M\times \mathbb{C}$, and by formalizing the $\mathbb{C}$- directions, a meromorphic relative connection $\mathcal{A}$ on $\hat{p}:TM\times \widehat{D}\to M \times \widehat{D}$. On the other hand, by Section \ref{trivJoyceSec} we have the trivial Joyce structure associated to $(M,\Omega,\Gamma,Z)$, whose $\mathbb{C}^*$-family of flat symplectic connections is given by 
\begin{equation*}
    \mathcal{A}^{\text{st},\epsilon}:=\mathcal{H}+\frac{1}{\epsilon}v\,.
\end{equation*}
As before, we obtain a meromorphic relative connection $\mathcal{A}^{\text{st}}$ on $\hat{p}:TM\times \widehat{D}\to M \times \widehat{D}$.\\

As explained in Section \ref{formal_setting}, we have a group $\widehat{\mathcal{G}}$ of formal gauge transformations of $\hat{p}$ and a corresponding action on connections (recall \eqref{formgaugeconn}). Following some of the ideas in \cite[Section 5.3]{KSresurgence}
we will show that one can find a formal gauge transformation $\hat{g}\in \widehat{\mathcal{G}}$ given by $\dot{g}\in \text{Lie}(\widehat{\mathcal{G}})$ such that 
\begin{equation}\label{maineq}
    e^{\mathcal{L}_{\dot{g}}}\mathcal{A}^{\text{st}}=\mathcal{A}\,.
\end{equation}

\begin{theorem}\label{mt1}
    Given a Joyce structure over $M$, there exists a unique $\dot{g}\in \text{Lie}(\widehat{\mathcal{G}})$ solving equation \eqref{maineq} and such that $\dot{g}|_{M}=0$.
\end{theorem}

\begin{proof}

We pick local coordinates $(Z^i)$ for $M$, and consider the induced coordinates $(Z^i,\theta^i)$ on $TM$. We  denote $\mathcal{A}_{i}:=\mathcal{A}_{\partial_{Z^i}}$ and $v_i:=v_{\partial_{Z^i}}$. We will build $\dot{g}$ by guaranteeing that \eqref{maineq} holds order by order in $\epsilon$. For $\dot{g}_1 \in \Gamma(TM,V_{\pi})$ we have the equation
\begin{equation*}
    e^{\mathcal{L}_{\epsilon\dot{g}_1}}\mathcal{A}^{\text{st}}_i=\mathcal{A}_i^{\text{st}}+[\dot{g}_1,v_i]+\mathcal{O}(\epsilon)\,,
\end{equation*}
where $\mathcal{O}(\epsilon)$ denotes terms of order $\epsilon$ or higher. If $\omega_i:=\mathcal{A}_i-\mathcal{A}_i^{\text{st}}$, then we would like to find $\dot{g}_1$ such that 
\begin{equation}\label{first}
    [\dot{g_1},v_i]=\omega_i, \quad i=1,...,\text{dim}_{\mathbb{C}}(M) \,.
\end{equation}
If we write in components $\dot{g}_1=\dot{g}_1^k\partial_{\theta^k}$ and $\omega_i=\omega_i^k\partial_{\theta^{k}}$, the using that $v_{i}=\partial_{\theta^i}$ the equation \eqref{first} reduces to
\begin{equation}\label{g1_eq}
    -\partial_{\theta^i}\dot{g}_1^k\partial_{\theta^k}=\omega_i^k\partial_{\theta^k}\,.
\end{equation}
The flatness of the Joyce structure implies that 
\begin{equation*}
    [\omega_i,v_j]+[v_i,\omega_j]=0 \implies \partial_{\theta^j}\omega_i^k=\partial_{\theta^i}\omega_j^k\,.
\end{equation*} Hence, a solution to \eqref{g1_eq} can be found by the (vector valued) path integral along fibers of $\pi: TM\to M$
\begin{equation}\label{g1_sol}
    \dot{g}_1(Z,\theta)=-\int_{\gamma_{(Z,\theta)}}\omega_i(Z,\varphi)\mathrm{d}\varphi^i
\end{equation}
where $\gamma_{(Z,\theta)}$ is a path from $(Z,0)$ to $(Z,\theta)$. The solution is clearly holomorphic in $Z$ and is independent of deformations of the path in the fiber. It vanishes on the zero section by construction and it is clear that $\dot{g}_1$ is unique with this property.  \\

By Lemma \ref{flatgaugelemma}, we now have the flat connection
\begin{equation*}
    \mathcal{A}_i^{(1)}:=e^{\mathcal{L}_{\epsilon\dot{g}_1}}\mathcal{A}^{\text{st}}_i=\mathcal{A}_i+\epsilon\omega_i^{(1)}+\mathcal{O}(\epsilon^2)\,,
\end{equation*}
where $\omega_i^{(1)}$ is a vertical vector field with respect to $\pi:TM \to M$. 
Flatness of $\mathcal{A}$ and $\mathcal{A}^{(1)}$ implies that
\begin{equation}\label{g2_flat}
    [v_j,\omega_i^{(1)}]=[v_i,\omega_j^{(1)}] \implies\partial_{\theta^j}\omega_i^{(1)}=\partial_{\theta^i}\omega_j^{(1)}\,.
\end{equation}
We now want to find $\dot{g}_2\in \Gamma(TM,V_{\pi})$ such that
\begin{equation*}
    e^{\mathcal{L}_{\epsilon^2\dot{g}_2}}\mathcal{A}_i^{(1)}=\mathcal{A}_i+\mathcal{O}(\epsilon^2)\,.
\end{equation*}
Since
\begin{equation*}
    \mathcal{A}_i^{(2)}:=e^{\mathcal{L}_{\epsilon^2\dot{g}_2}}\mathcal{A}_i^{(1)}=\mathcal{A}_i+\epsilon\omega_i^{(1)}+\epsilon[\dot{g}_2,v_i]+\mathcal{O}(\epsilon^2)
\end{equation*}
we want to find $\dot{g}_2$ solving the system 
\begin{equation*}
    [v_i,\dot{g}_2]=\omega_i^{(1)}\,, \quad i=1,...,\text{dim}_{\mathbb{C}}(M). 
\end{equation*}
As before, equation \eqref{g2_flat} implies that we can solve for $\dot{g}_2$ by using fiber-wise path integrals. We can choose it to vanish along the zero section by the choice of integration path as in \eqref{g1_sol}, and $\dot{g}_2$ is again unique with these properties. \\

In general, assume that we have the flat connection 
\begin{equation*}
    \mathcal{A}_i^{(k)}=\mathcal{A}_i+\epsilon^k\omega_i^{(k)}+\mathcal{O}(\epsilon^{k+1})\,.
\end{equation*}
where $\omega_i^{(k)}$ is a vertical vector field. Flatness of $\mathcal{A}$ and $\mathcal{A}^{(k)}$ implies
\begin{equation*}
    [v_j,\omega_i^{(k)}]=[v_i,\omega_j^{(k)}] \implies\partial_{\theta^j}\omega_i^{(k)}=\partial_{\theta^i}\omega_j^{(k)}\,,
\end{equation*}
while
\begin{equation*}
    e^{\mathcal{L}_{\epsilon^{k+1}\dot{g}_{k+1}}}\mathcal{A}_i^{(k)}=\mathcal{A}_i+\epsilon^{k}\omega_i^{(k)}+\epsilon^k[\dot{g}_{k+1},v_i]+\mathcal{O}(\epsilon^{k+1})\,.
\end{equation*}
Hence, we want to find $\dot{g}_{k+1}$  such that
\begin{equation*}
    [v_i,\dot{g}_{k+1}]=\omega_i^{(k)}, \quad i =1,2...,\text{dim}_{\mathbb{C}}(M)\,,
\end{equation*} and we can solve as before and obtain $\dot{g}_{k+1}$ via the path integral

\begin{equation}\label{piformula}
    \dot{g}_{k+1}(Z,\theta)=\int_{\gamma(Z,\theta)}\omega_i^{(k)}(Z,\varphi)\mathrm{d}\varphi^i
\end{equation}
vanishing on the zero section, and $\dot{g}_{k+1}$ is unique with this property. \\

While we worked in local coordinates $(Z^i)$ on $M$ (and the induced $(Z^i,\theta^i)$ on $TM$) we can use the uniqueness of the local $\dot{g}_{k}$ to glue them together into a global $\dot{g}_k$.\\

Applying all the previous gauge transformations successively we obtain
\begin{equation*}
    \left(\prod_{k=1}^{\infty}e^{\mathcal{L}_{\epsilon^k\dot{g}_k}}\right)\mathcal{A}^{\text{st}}=\mathcal{A}\,,
\end{equation*}
where each new factor is applied to the left of what resulted in the action of the previous factor. The way to make sense of the above equation is as follows: for any truncation of the powers of $\epsilon^k$ on both sides of the equation, the equality holds after applying sufficiently many products of gauge transformations.  Furthermore, the infinite product can be interpreted as a gauge transformation by noticing that for any $n>0$, the terms of order $\epsilon^k$ for $k\leq n$ obtained via the Baker-Campbell-Hausdorff (BCH) formula stabilize for sufficiently many products. Namely, we can write 

\begin{equation*}
    \left(\prod_{k=1}^{\infty}e^{\mathcal{L}_{\epsilon^k\dot{g}_k}}\right)=e^{\mathcal{L}_{\dot{g}}}
\end{equation*}
where in the epsilon expansion of $\dot{g}$, the $\epsilon^k$-term is obtained by successively applying the BCH formula on the products until the $\epsilon^k$-term stabilizes.\\ 

The fact that the resulting $\dot{g}$ vanishes on the zero section of $\pi:TM\to M$ follows from the fact that vertical vector fields vanishing along the zero section are closed under linear combinations and Lie brackets, together with the BCH formula. 

\end{proof}

\begin{remark}
    Note that while the boundary condition $\dot{g}|_M=0$ is natural, one might be interested in a $\dot{g}$ gauging $\mathcal{A}^{\text{st}}$ to $\mathcal{A}$ but not vanishing on the zero section. We will see an instance of this in the examples of Section \ref{exsec}.
\end{remark}
It is useful to explicitly write the equations that must hold for any $\dot{g}\in \text{Lie}(\widehat{\mathcal{G}})$ satisfying \eqref{maineq}, order by order in $\epsilon$.  We collect the equations in the following proposition, using the notation from the proof of Theorem \ref{mt1}.

\begin{proposition}\label{geq}
    Consider $\dot{g}=\sum_{k=1}^{\infty}\epsilon^k\dot{g}_k\in \text{Lie}(\widehat{\mathcal{G}})$. Then $\dot{g}$ satisfies \eqref{maineq} if and only if, with respect to flat Darboux coordinates $(Z^i)$ on $M$
    \begin{equation}\label{firstrec}
    [\dot{g}_1,v_i]=\omega_i, \quad i=1,2,...,\text{dim}_{\mathbb{C}}(M)
\end{equation}
and for each $l>0$ the following equation holds for $i=1,2,...,\text{dim}_{\mathbb{C}}(M)$

\begin{equation}\label{receq}
0=\sum_{r=1}^{l+1}\frac{1}{r!}\sum_{k_1+...+k_r=l+1}[\dot{g}_{k_1},...,[\dot{g}_{k_r},v_i]...] + \sum_{r=1}^{l}\frac{1}{r!}\sum_{k_1+...+k_r=l}[\dot{g}_{k_1},...,[\dot{g}_{k_r},\partial_{Z^i}],...]    
\end{equation}

\end{proposition}

\begin{proof}
    This follows by simply expanding \eqref{maineq} as a formal series in $\epsilon$, and collecting terms with the same power of $\epsilon$, together with the fact that $\mathcal{H}_{\partial_{Z^i}}=\partial_{Z^i}$ for $\nabla$-flat coordinates.
\end{proof}
For example, for $l=1,2$ the equation \eqref{receq} reduces to
\begin{equation}\label{firstreceq}
\begin{split}
&[\dot{g}_2,v_i]+[\dot{g}_1,\partial_{Z^i}]+\frac{1}{2}[\dot{g}_1,[\dot{g}_1,v_i]]=0,\\
&[\dot{g}_3,v_i]+[\dot{g}_2,\partial_{Z^i}]+\frac{1}{2}\left([\dot{g}_1,[\dot{g}_2,v_i]]+ [\dot{g}_2,[\dot{g}_1,v_i]]+[\dot{g}_1,[\dot{g}_1,\partial_{Z^i}]]\right)+\frac{1}{6}[\dot{g}_1,[\dot{g}_1,[\dot{g}_1,v_i]]]=0\,.\\
\end{split}
\end{equation}
\begin{remark}\label{g1hamrem}
    When we write $\mathcal{A}^{\epsilon}$ in terms of the Pleba\'nski function $W$ as in \eqref{plebform}, we see that \eqref{firstrec} can be written as
    \begin{equation*}
        -[v_i,\dot{g}_1] = \text{Ham}^{\Omega^v}(v_iW)\,.
    \end{equation*}
    Choosing Darboux coordinates $(Z^i)$ on $M$ and the induced coordinates $(Z^i,\theta^i)$ on $TM$ we can write the previous equation as 
    \begin{equation*}
        -\partial_{\theta^i}(\dot{g}_1^k)\partial_{\theta^k}=\Omega^{jk}\frac{\partial^2W}{\partial \theta^i\partial \theta^j}\partial_{\theta^k}
    \end{equation*}
    where $\Omega^{ij}$ are the coefficients of $\Omega^{-1}$. Hence it is natural to take
    \begin{equation}\label{g1pleb}
        \dot{g}_1=-\text{Ham}^{\Omega^v}(W)\,.
    \end{equation}
    If the Pleba\'nski function $W$ is chosen to satisfy \eqref{plebzerosec}, then the resulting $\dot{g}_1$ satisfies $\dot{g}_1|_M=0$. This and the equations in \eqref{receq} show that $W$ can be used to determine all subsequent $\dot{g}_k$, and these are unique if we impose that $\dot{g}_k|_M=0$.
\end{remark}
\subsection{Formal twistor coordinates for Joyce structures}\label{formaldarbouxcoordsec}

Recall from Section \ref{twistorJoycesec} that we have a twistor space $(p:\mathcal{Z}\to \mathbb{P}^1, \varpi)$ associated to a Joyce structure over $M$. Since $\varpi$ restricts to a holomorphic symplectic form on each fiber $\mathcal{Z}_{\epsilon}=p^{-1}(\epsilon)$, it is a natural question to look for holomorphic functions on $\mathcal{Z}$ that restrict to Darboux coordinates for $\varpi^{\epsilon}:=\varpi|_{\mathcal{Z}_{\epsilon}}$.\\

If $n = \text{dim}_{\mathbb{C}}(M)$, then   $\text{dim}_{\mathbb{C}}(\mathcal{Z})=n+1$ and $\text{dim}_{\mathbb{C}}(\mathcal{Z}_{\epsilon})=n$. We therefore want to look for $n$ functions on $\mathcal{Z}$ that restrict to Darboux coordinates of $\varpi^{\epsilon}$.\\

In the case of the trivial Joyce structure for Section \ref{trivJoyceSec}, there are natural candidates for this. Indeed, if $(Z^i)$ are $\nabla$-flat Darboux coordinates on $U\subset M$, and $(Z^i,\theta^i)$ the induced coordinates on $\pi^{-1}(U)\subset TM$, we can define the following functions\footnote{It is also natural to define $x^{i,\text{st}}$ by $x^{i,\text{st}}=\epsilon \theta^i-Z^i$ in order to have an expression in $\mathcal{O}_{TM}[[\epsilon]]$ instead of $\epsilon^{-1}\cdot \mathcal{O}_{TM}[[\epsilon]]$.} on $\pi^{-1}(U)\times \mathbb{C}^{*}\subset TM\times \mathbb{P}^1$:
\begin{equation}\label{stcoords}
    x^{i,\text{st}}:=\theta^i-\frac{Z^i}{\epsilon}, \quad i=1,..,n\,.
\end{equation}
These functions satisfy
\begin{equation}
    \mathcal{A}^{\text{st}}_{\partial_{Z^j}}(x^{i,\text{st}})=0, \quad i,j=1,...,n\,.
\end{equation}
and hence descend to the twistor space $\mathcal{Z}^{\text{st}}$. Furthermore, an easy computation shows that (see \eqref{twisttriv} and the footnote) for a fixed $\epsilon \in \mathbb{C}^{*}$
\begin{equation}\label{trivtwistdarboux}
    q^*\varpi^{\text{st}}|_{TM\times \{\epsilon\}}=\epsilon^{2}\Omega_{ij}\mathrm{d}x^{i,\text{st}}\wedge \mathrm{d}x^{j,\text{st}}
\end{equation}
where $\Omega_{ij}\in \mathbb{C}$ are the coefficients of $\Omega$ in the coordinates $(Z^i)$. Hence, the functions $x^{i,\text{st}}$ descend to Darboux coordinates for $\varpi^{\text{st},\epsilon}$.\\

We now fix a Joyce structure $(M,\Omega,\Gamma,Z,h)$ with corresponding family of connections $\mathcal{A}^{\epsilon}$, and also consider the trivial Joyce structure $(M,\Omega,\Gamma,Z,\mathcal{H})$ with family of connections $\mathcal{A}^{\text{st},\epsilon}$. We consider the corresponding formal relative connections $\mathcal{A}$ and $\mathcal{A}^{\text{st}}$ on $\hat{p}:TM\times \widehat{D}\to M\times \widehat{D}$, and any $\dot{g}\in \text{Lie}(\widehat{\mathcal{G}})$ satisfying 
\begin{equation*}
e^{\mathcal{L}_{\dot{g}}}\mathcal{A}^{\text{st}}=\mathcal{A}\,.
\end{equation*}
We also fix coordinates $(Z^i, \theta^i)$ as before on $\pi^{-1}(U)\subset TM$. 

\begin{proposition}\label{flatcoordsprop}
    Under the previous assumptions, define $x^i\in \epsilon^{-1}\cdot \mathcal{O}_{TM}(\pi^{-1}(U))[[\epsilon]]$ by 
    \begin{equation*}
        x^i:=e^{\dot{g}}x^{i,\text{st}}=\sum_{k=0}^{\infty}\frac{\dot{g}^k(x^{i,\text{st}})}{k!}\,.
    \end{equation*}
    Then the $x^i$ satisfy 
\begin{equation*}
    \mathcal{A}_{\partial_{Z^j}}(x^i)=0\, \quad \text{for} \quad  i,j=1,...,n.
\end{equation*}
\end{proposition}

\begin{proof}
    We let $\mathcal{A}_j:=\mathcal{A}_{\partial_{Z^j}}$.
    For ease of notation, denote by $B$ the operator $B(\mathcal{A}_j^{\text{st}})=[\dot{g},\mathcal{A}_j^{\text{st}}]$. Then 
    \begin{equation*}
        \mathcal{A}=e^{\mathcal{L}_{\dot{g}}}\mathcal{A}^{\text{st}}_j=\sum_{k=0}^{\infty}\frac{B^k(\mathcal{A}^{\text{st}}_j)}{k!}\,.
    \end{equation*}
    One can easily check by induction that 
    \begin{equation*}
        B^k(\mathcal{A}_j^{\text{st}})=\sum_{l=0}^{k}(-1)^{l}\binom{k}{l}\dot{g}^{k-l}\mathcal{A}^{\text{st}}_j\dot{g}^l\,.
    \end{equation*}
    If we group the terms of $e^{\mathcal{L}_{\dot{g}}}\mathcal{A}_j^{\text{st}}(e^{\dot{g}}x^{i,\text{st}})$ with the same factor $\dot{g}^{p_1}\mathcal{A}^{\text{st}}_j\dot{g}^{p_2}$ we obtain

    \begin{equation*}
        \left(\sum_{k=p_1}^{p_1+p_2}\frac{1}{k!}\frac{(-1)^{k-p_1}}{(p_2+p_1-k)!}\binom{k}{k-p_1}\right)\dot{g}^{p_1}\mathcal{A}^{\text{st}}_j\dot{g}^{p_2}x^{i,\text{st}}\,,
    \end{equation*}
    Which can be rewritten as
    \begin{equation*}
        \frac{1}{p_1!p_2!}\left(\sum_{k=0}^{p_2}(-1)^k\binom{p_2}{k}\right)\dot{g}_1^{p_1}\mathcal{A}_j^{\text{st}}\dot{g}_2^{p_2}x^{i,\text{st}}\,.
    \end{equation*}
    In the cases where $p_2=0$ the above vanishes because $\mathcal{A}^{\text{st}}_jx_i^{\text{st}}=0$, while for the cases $p_2>0$ we use
    \begin{equation*}
        \sum_{k=0}^{p_2}(-1)^k\binom{p_2}{k}=0
    \end{equation*}
    which follows from setting $a=1$, $b=-1$ in 
\begin{equation*}
    (a+b)^{p_2}=\sum_{k=0}^{p_2}\binom{p_2}{k}a^{p_2-k}b^k\,.
\end{equation*}
The desired result then follows.
\end{proof}
While we will not discuss the notion of formal twistor space in detail,  the result of the previous proposition can be used to produce several statements related to this notion, which we summarize in the remark below. 
\begin{remark}
\begin{itemize}
\item The statement that $\mathcal{A}(x^i)=0$ can be interpreted as the fact that the functions $\epsilon \cdot x^i\in \mathcal{O}_{TM}(\pi^{-1}(U))[[\epsilon]]$ descend to functions on the formal twistor space $\widehat{\mathcal{Z}}$, obtained by doing the formal completion of $\mathcal{Z}$ along the fiber $\mathcal{Z}_{0}$ over $\epsilon =0$. 
\item Thinking of \eqref{trivtwistdarboux} as an expression in $\Omega^2(TM)[[\epsilon]]$, we have that
\begin{equation}\label{formalsympgauge}
    e^{\mathcal{L}_{\dot{g}}}(\epsilon^{2}\Omega_{ij}\mathrm{d}x^{i,\text{st}}\wedge \mathrm{d}x^{j,\text{st}})=\epsilon^{2}\Omega_{ij}\mathrm{d}(e^{\dot{g}}x^{i,\text{st}})\wedge \mathrm{d}(e^{\dot{g}}x^{j,\text{st}})=\epsilon^{2}\Omega_{ij}\mathrm{d}x^i\wedge \mathrm{d}x^j\,,
\end{equation}
where the exterior derivative acts trivially on the formal variable $\epsilon$. The right hand side of \eqref{formalsympgauge} is clearly closed and vanishes on $\text{Im}(\mathcal{A})$ by construction, and hence descends to $\widehat{\mathcal{Z}}$. We can then think of $x^i$ as formal twistor Darboux coordinates. 
\item Using \eqref{g1pleb} it is not hard to check that in the $(Z^i,\theta^i)$ coordinates
\begin{equation}\label{holsympexp}
    \Omega_{ij}\mathrm{d}x^i\wedge \mathrm{d}x^j=\Omega_{ij}\mathrm{d}x^{i,\text{st}}\wedge \mathrm{d}x^{j,\text{st}}+2\frac{\partial^2W}{\partial \theta^i\partial \theta^j} \mathrm{d}\theta^i\wedge \mathrm{d}Z^j + 2\frac{\partial^2W}{\partial \theta^j \partial Z^i}\mathrm{d}Z^i\wedge \mathrm{d}Z^j+\mathcal{O}(\epsilon)\,,
\end{equation}
where $\mathcal{O}(\epsilon)$ are terms of order $\epsilon$ or higher. Using the Pleba\'nski second heavenly equation satisfied by $W$, one checks that up to $\mathcal{O}(\epsilon)$-terms \eqref{holsympexp} reproduces the (pullback by $q$ of the) twistor family of symplectic forms $\varpi$ from \cite[Equation (24-26)]{TwistorJoyce}. It is not obvious to the author whether in general the $\mathcal{O}(\epsilon)$-terms vanish order by order or not. In the A1 example from Section \ref{exsec} it will happen that all higher order terms in $\epsilon$ vanish. 
\end{itemize}
\end{remark}

We end this section with an explicit formula for $x^i$ in the case that $\dot{g}$ can be written in Hamiltonian form. 

\begin{corollary} In the same setting as the previous proposition, assume there is a function $\widehat{W}\in \mathcal{O}_{TM}[[\epsilon]]$

\begin{equation*}  \widehat{W}=\sum_{k=1}^{\infty}\epsilon^kW_k
\end{equation*}
such that \begin{equation*}
    \dot{g}=\text{Ham}^{\Omega^v}(\widehat{W}).
\end{equation*}
Then we have the formula

\begin{equation}\label{darbouxform}
    x^i=\theta^i-\epsilon^{-1}z^i+\eta^{pi}\sum_{k=1}^{\infty}\epsilon^{k}\sum_{r=1}^{k}\frac{1}{r!}\sum_{l_1+...+l_r=k}\{W_{l_1},\{W_{l_2},...,\{W_{l_{r-1}},\frac{\partial W_{l_r}}{\partial \theta^p}\}...\}
\end{equation}
where $-W_1=W$ is the Pleba\'nski function, the Poisson brakets are taken with respect to $\Omega^{v}$, and for the case $r=1$ there are no Poisson brackets but only the term $\frac{\partial W_k}{\partial \theta^p}$.
    
\end{corollary}
\begin{proof}

Formula \eqref{darbouxform} follows easily from writing $e^{\dot{g}}x^{i,\text{st}}$ order by order in $\epsilon$ and using the fact that $x^i=e^{\dot{g}}x^{i,\text{st}}$, while the fact that $W_1=-W$ follows from \eqref{g1pleb}.
\end{proof}
\section{Towards resurgence of formal gauge transformations}\label{sec4}

In this section we make the first steps towards showing that infinitesimal formal gauge transformations $\dot{g}\in \text{Lie}(\widehat{\mathcal{G}})$ gauging the trivial Joyce structure to a given Joyce structure \eqref{maineq} are resurgent. Our main result shows that the Borel transform 
\begin{equation}\label{BTgauge}
    \mathcal{B}[\dot{g}](\xi):=\sum_{k=1}^{\infty}\frac{\dot{g}_k}{(k-1)!}\xi^{k-1}
\end{equation}
converges on a neighbourhood of $\xi=0$, locally uniformly in the parameters $p\in TM$.  We also show a similar result regarding the Borel transform of the formal twistor Darboux coordinates from Section \ref{formaldarbouxcoordsec}. The remaining step, which requires to show that $\mathcal{B}[\dot{g}](\xi)$ admits ``endless analytic continuation" is much harder, and we leave for future work. Here ``endless analytic continuation" means that we must show that there is some closed discrete subset $S\subset \mathbb{C}$ such that $\mathcal{B}[\dot{g}](\xi)$ admits analytic continuations in $\xi$ along paths in $\mathbb{C}-S$ starting from the neighbourhood of convergence of \eqref{BTgauge}. \\

As explained in the introduction,  showing that $\dot{g}$ is resurgent would allow us to study the Stokes automorphisms of $\dot{g}$ (which are expected to be related to DT-invariants), as well as the Borel summations of $\dot{g}$ (if they exist). Furthermore, one could also study whether objects associated to $\dot{g}$ are resurgent, like the formal twistor Darboux coordinates from Section \ref{formaldarbouxcoordsec} and the corresponding formal $\tau$-functions \cite{BridgelandTau}.
\subsection{Borel transform of formal infinitesimal gauge transformations}\label{btgaugesec}

Throughout this section we fix a Joyce structure $(M,\Omega,\Gamma,Z,h)$ with associated relative connection $\mathcal{A}$ on $\widehat{p}:TM\times \widehat{D}\to M\times \widehat{D}$, and consider an infinitesimal gauge transformation $\dot{g}\in \text{Lie}(\widehat{\mathcal{G}})$ satisfying
\begin{equation*}
    e^{\mathcal{L}_{\dot{g}}}\mathcal{A}^{\text{st}}=\mathcal{A}\,.
\end{equation*}
For example we could take $\dot{g}$ as in Theorem \ref{mt1}. Recall that by definition $\dot{g}\in \text{Lie}(\widehat{\mathcal{G}})$ is a formal series

\begin{equation*}
    \dot{g}=\sum_{k=1}^{\infty}\dot{g}_k\epsilon^k
\end{equation*}
where $\dot{g}_k$ are sections of the vertical bundle $V_{\pi}\to TM$ associated to the canonical projection $\pi:TM \to M$. Its Borel transform is defined by \eqref{BTgauge}, and our aim is to show that \eqref{BTgauge} converges in a neighbourhood of $\xi=0$, locally uniformly in parameters $p\in TM$.\\

For $p=(p^i)\in \mathbb{C}^n$ will denote by $B_r(p):=\prod_{i=1}^{n}D(p^i,r)$ the polydisc in $\mathbb{C}^n$ where each $D(p^i,r)\subset \mathbb{C}$ is a disk centred at $p^i\in \mathbb{C}$ of radius $r$. Sometimes we will abbreviate $B_r(p)$ with just $B_r$.

\begin{lemma}\label{prelem} Consider two concentric polydiscs $B_r\subset B_R\subset \mathbb{C}^n$ with radii $0<r<R$. Furthermore, consider $k\geq 2$ vector fields $V_i=V_i^j\partial_{Z^j}$ holomorphic on a neighborhood $U$ of $B_R$, where $i=1,...,k$, and define

\begin{equation}\label{vectnorm}
    ||V_i||_{K}:=\text{max}_j(\text{sup}_{z\in K}|V_i^j(z)|)\,.
\end{equation}
for $K\subset U.$ Then 
\begin{equation}\label{nestedcom}
    ||\;[V_1,[V_2,...,[V_{k-1},V_k],...]\;||_{B_r}\leq \left(\frac{2n(k-1)}{R-r}\right)^{k-1}\prod_{i=1}^k||V_i||_{B_R}\,.
\end{equation}
\end{lemma}
\begin{proof}
    First fix $t\in (0,R)$ and $\rho\in (0,R-t]$. For $k=2$ and $B_t$ a ball of radius $t$ concentric with $B_R$, we have for any $z\in B_t$ the following estimate on the components of $[V_1,V_2]=[V_1,V_2]^j\partial_{Z^j}$
    \begin{equation*}
        | \; [V_1,V_2]^j(z) \; |\leq \sum_{l=1}^{n}|V_1^l(z)||\partial_{Z^l}V_2^j(z)|+|V_2^l(z)||\partial_{Z^l}V_1^j(z)|\,.
    \end{equation*}
    Applying the Cauchy estimate on the terms with derivatives one finds
    \begin{equation*}
        | \; [V_1,V_2]^j(z) \; |\leq \sum_{l=1}^{n}\left(||V_1||_{B_{t+\rho}}\frac{||V_2||_{B_{t+\rho}}}{\rho}+||V_2||_{B_{t+\rho}}\frac{||V_1||_{B_{t+\rho}}}{\rho}\right)=\frac{2n}{\rho}||V_1||_{B_{t+\rho}}||V_2||_{B_{t+\rho}}\,,
    \end{equation*}
    from which it follows that 

    \begin{equation}\label{basecase}
        ||[V_1,V_2]||_{B_t}\leq \frac{2n}{\rho}||V_1||_{B_{t+\rho}}||V_2||_{B_{t+\rho}}\,.
    \end{equation}

    We now prove \eqref{nestedcom} by sucessively using \eqref{basecase} for appropriate choices of $t$ and $\rho$. Note that the case $k=2$ is given already by \eqref{basecase} if we take $t=r$ and $R=t+\rho$. Let $\rho=\frac{R-r}{k-1}$ and let $t_j=r+\rho j$ for $j=0,1,...,k-1$, so that $t_0=r$ and $t_{k-1}=R\,$. By repeatedly applying \eqref{basecase} by picking $t$ and $t+\rho$ to be consecutive $t_j$'s we find 

    \begin{equation*}
        ||[V_1,[V_2,...,[V_{k-1},V_k],...]||_{B_{t_0}}\leq \frac{2n}{\rho}||V_{1}||_{B_{t_1}}\cdot ||[V_2,...,[V_{k-1},V_k],...]||_{B_{t_1}}\leq \left(\frac{2n}{\rho}\right)^{k-1}\prod_{i=1}^{k}||V_{i}||_{B_{t_{k-1}}}
    \end{equation*}
    from which the result follows.
\end{proof}
We can now combine the previous lemma with \eqref{receq} to obtain the main estimate that we will need.
\begin{corollary}\label{keycol} Assume that we work locally on a coordinate chart $U\subset TM$ and consider $B_r\subset B_R\subset U$ as in the previous lemma. For $l\geq 1$ the coefficients $\dot{g}_{k}$ of $\dot{g}=\sum_{k=1}^{\infty}\dot{g}_k\epsilon^k$ satisfy

\begin{equation*}
    ||\;[v_i,\dot{g}_{l+1}]\;||_{B_r} \leq \sum_{p=2}^{l+1}\frac{1}{p!}\left(\frac{2np}{R-r}\right)^p\sum_{k_1+...+k_p=l+1}\prod_{j=1}^{p}||\dot{g}_{k_j}||_{B_R}+\sum_{p=1}^{l}\frac{1}{p!}\left(\frac{2np}{R-r}\right)^p\sum_{k_1+...+k_p=l}\prod_{j=1}^{p}||\dot{g}_{k_j}||_{B_R}
\end{equation*}
for all $i=1,...,n$.
\end{corollary}
\begin{proof}
    One can first isolate the $[v_i,\dot{g}_{l+1}]$ term from the first sum in \eqref{receq} and take it to the other side of the equality. The required inequality then follows immediately from the previous lemma and the fact that $||v_i||_{B_R}=||\partial_{\theta^i}||_{B_R}=||\partial_{Z^i}||_{B_R}=1$.
\end{proof}

\begin{theorem}\label{mt2}
The Borel transform $\mathcal{B}[\dot{g}](\xi)$ converges in a neighbourhood of $\xi=0$ locally uniformly in $p\in TM$ if and only the same holds for $B[\dot{g}][\xi]$ when restricted to the zero section $M\subset TM$. In particular, the Borel transform of the $\dot{g}$ built in Theorem \ref{mt1} converges in a neighbourhood of $\xi=0$ locally uniformly in $p\in TM$.
\end{theorem}

\begin{proof}
    We consider first the case where $\dot{g}$ is the one built from Theorem \ref{mt1}, so that $\dot{g}|_M=0$. Furthermore, we take  $p\in M\subset TM$, a coordinate neighbourhood $U\subset TM$ of $p$ with coordinates $(Z^i,\theta^i)$ induced from coordinates $(Z^i)$ on the base $M$, and a polydisc $B_r(p)\subset U$. We will show that $\mathcal{B}[\dot{g}]$ converges in a neighbourhood of $\xi=0$, uniformly in $(Z^i,\theta^i)\in B_r$, and at the end mention how the proof can be extended to the more general cases in the statement of the theorem. \\
    
    Given a polydisc $B_R(p)\subset U$ with $r<R$, we will use the notation $a_k=|| \; \dot{g}_k \; ||_{B_r},$ $A_k=||\; \dot{g}_k\;||_{B_R} $,  and $C=\frac{2n}{R-r}$ throughout the proof.  To show what we want, it is enough to show that
    \begin{equation}\label{bt1}
        \sum_{k=1}^{\infty}\frac{a_k}{k!}\xi^{k}
    \end{equation} converges for sufficiently small $\xi$. Since $[v_i,\dot{g}_{l+1}]=\partial_{\theta^{i}}(\dot{g}_{l+1}^{k})\partial_{\theta^k}$ and $\dot{g}_{l+1}$ vanishes on the zero section,  we can estimate $\dot{g}_{l+1}(z)$ in terms of its vertical derivatives via a path integral in the vertical directions. Namely, we can  write
    \begin{equation}\label{pathintest}
        \dot{g}_{l+1}(Z,\theta)=\int_{\gamma(Z,\theta)}\partial_{\theta^i}\dot{g}(Z, \varphi)\mathrm{d}\varphi^i
    \end{equation}
    where $\gamma(Z,\theta)$ is a straight path from the $(Z,0)$ to $(Z,\theta)$, and then conclude that   \begin{equation*}
    a_{l+1}=||g_{l+1}||_{B_r}\leq r\cdot n \cdot \max\{||\;[v_i, g_{l+1}]\;||_{B_r}\}_{i=1}^{n}\,.
\end{equation*} 
Combining this inequality with the inequality from Corollary \ref{keycol} we find
\begin{equation}\label{in1}
    a_{l+1} \leq r\cdot n\left(\sum_{p=2}^{l+1}\frac{1}{p!}C^pp^p\sum_{k_1+...+k_p=l+1}\prod_{i=1}^{p}A_{k_i}+\sum_{p=1}^{l}\frac{1}{p!}C^pp^p\sum_{k_1+...+k_p=l}\prod_{i=1}^{p}A_{k_i}\right)\,.
\end{equation}
Pick $y_1>A_1$ and define  $y_{l+1}$ for $l\geq 1$ as the right hand side of the inequality \eqref{in1} with all $A_{k_i}$ replaced by the $y_{k_i}$. Note that this definition makes sense since the $y_{k_i}$ that appear have $k_i<l+1$. It then follows that $a_{k}<y_{k}$ for all $k\geq 1$. Hence, to show that \eqref{bt1} converges in a neighbourhood of $\xi=0$ it is enough to show that 
\begin{equation*}
    C(\xi):=\sum_{k=1}^{\infty}\frac{y_k}{k!}\xi^{k}=\sum_{k\geq 1}c_k\xi^{k-1}
\end{equation*}
converges in a neighbourhood of $\xi=0$. To continue we will need the following combinatorial lemma.\\

\begin{lemma}Let $p$ be a positive integer. If $k_i\geq 1$ for $i=1,...,p$ are integers satisfying $k_1+...+k_p=m$ then
\begin{equation*}
    p!\leq \frac{m!}{k_1!...k_p!}\,.
\end{equation*}
\end{lemma}
\begin{proof}
Consider first two integers $a,b\geq 1$. It is then easy to see by a direct computation that $(a+b-1)!\geq a!b!$. Successively applying this we find
\begin{equation*}
k_1!...k_p!\leq (k_1+...+k_p-p+1)!=(m-p+1)!\,.
\end{equation*}
We then obtain the desired inequality
\begin{equation*}
\frac{m!}{k_1!...k_p!}\geq \frac{m!}{(m-p+1)!}=m(m-1)...(m-p+2)\geq p(p-1)...2=p!\,.
\end{equation*}
\end{proof}
Applying the lemma for the cases $m=l+1$ and $m=l$, and using the notation $[\xi^k]C(\xi)^p$ for the coefficient of $\xi^k$ in $C(\xi)^p$. we obtain the following 
\begin{equation}\label{in2}
\begin{split}
    \frac{y_{l+1}}{(l+1)!}&=r\cdot n\left(\sum_{p=2}^{l+1}\frac{(Cp)^p}{(l+1)!p!}\sum_{k_1+...+k_p=l+1}y_{k_1}\cdot...\cdot y_{k_p}+\frac{1}{(l+1)}\sum_{p=1}^{l}\frac{(Cp)^p}{l!p!}\sum_{k_1+...+k_p=l}y_{k_1}\cdot...\cdot y_{k_p}\right)\\
    &\leq r\cdot n\left(\sum_{p=2}^{l+1}\frac{(Cp)^p}{(p!)^2}\sum_{k_1+...+k_p=l+1}\frac{y_{k_1}\cdot...\cdot y_{k_p}}{k_1!...k_p!}+\frac{1}{(l+1)}\sum_{p=1}^{l}\frac{(Cp)^p}{(p!)^2}\sum_{k_1+...+k_p=l}\frac{y_{k_1}\cdot...\cdot y_{k_p}}{k_1!...k_p!}\right)\\
    &\leq r\cdot n\left(\sum_{p=2}^{l+1}\frac{(Cp)^p}{(p!)^2}[\xi^{l+1}]C(\xi)^p+\frac{1}{(l+1)}\sum_{p=1}^{l}\frac{(Cp)^p}{(p!)^2}[\xi^l]C(\xi)^p\right)\,.\\
\end{split}
\end{equation}
We can encode \eqref{in2} for all $l\geq 1$ by writing 
\begin{equation}\label{in3}
    C(\xi)-y_1\xi \leq r\cdot n\sum_{l\geq 1}\xi^{l+1}\left(\sum_{p=2}^{l+1}\frac{(Cp)^p}{(p!)^2}[\xi^{l+1}]C(\xi)^p+\frac{1}{(l+1)}\sum_{p=1}^{l}\frac{(Cp)^p}{(p!)^2}[\xi^l]C(\xi)^p\right)\,,
\end{equation}
where $\leq$ denotes the partial order in  $\mathbb{R}[[\xi]]$ where for all  $k\geq 0$ the coefficient of $\xi^k$ on the left hand side is $\leq$ than the coefficient of $\xi^k$ on the right hands side. Now note that since $C(\xi)$ does not have a constant term, for $p>k$ we have $[\xi^k]C(\xi)^p=0$. We can then extend sums in $p$ in \eqref{in3} to infinity without changing anything.  
Swapping the order of sums we can then rewrite the above as 

\begin{equation}\label{in4}
    C(\xi)-y_1\xi \leq \sum_{p=2}^{\infty}\beta_pC(\xi)^p+\int_{0}^{\xi}\sum_{p=1}^{\infty}\beta_pC(\zeta)^p\mathrm{d}\zeta\,, \quad \beta_p=rn\frac{(Cp)^p}{(p!)^2}\,.
\end{equation}
Using that differentiation with respect to $\xi$ preserves the partial order $\leq$ on $\mathbb{R}[[\xi]]$, we find from differentiating \eqref{in4} and rearranging that 
\begin{equation*}
    C'(\xi)(1-\sum_{p\geq 2}p\beta_pC(\xi)^{p-1})\leq y_1+\sum_{p\geq 1}\beta_pC(\xi)^p\,.
\end{equation*}
To simplify the expressions we introduce the notation $\Phi_1(u):=\sum_{p\geq 2}p\beta_pu^{p-1}$ and $\Phi_2(u):=\sum_{p \geq 1}\beta_pu^p$, so that 

\begin{equation*}
    C'(\xi)(1-\Phi_1(C(\xi))\leq y_1+\Phi_2(C(\xi))\,.
\end{equation*}
Note that since $(1-\Phi_1(C(\xi)))|_{\xi=0}=1$, the  formal series $(1-\Phi_1(C(\xi)))$ is invertible in $\mathbb{R}[[\xi]].$ Furthermore $(1-\Phi_1(C(\xi)))^{-1}$ is a formal series with positive coefficients, since it is given by the formal series $\sum_{k=0}^{\infty}(\Phi_1(C(\xi)))^k$ and $\Phi_1(C(\xi))$ has positive coefficients. Since products of formal series with positive coefficients preserves the order of formal series, we obtain that 
\begin{equation}\label{in5}
    C'(\xi)\leq(1-\Phi_1(C(\xi))^{-1}(y_1+\Phi_2(C(\xi)))\,.
\end{equation}
The important thing to notice now is that the formal series $\Phi_1(u)$ and $\Phi_2(u)$ both converge. In fact they have an infinite radius of convergence. This can be seen by the Stirling approximation: as $p\to \infty$
\begin{equation}
(\beta_p)^{1/p}=(r\cdot n)^{1/p}\frac{Cp}{p!^{2/p}}\sim \frac{C e^2}{(2\pi p)^{1/p}p}\to 0\,,
\end{equation}
which shows that $\Phi_2(u)$ has infinite radius of convergence, and hence $\Phi_1(u)=\Phi_2'(u)-1$ as well. It follows that the series 
\begin{equation*}
    G(u):=(1-\Phi_1(u))^{-1}(y_1+\Phi_2(u))=\sum_{m\geq 0}g_mu^m\,.
\end{equation*}
has a positive radius of convergence, and $g_m \geq 0$ for all $m$ since it is a product of two formal series with positive coefficients.\\

Now consider the formal ODE
\begin{equation*}
    f'(\xi)=G(f(\xi)), \quad f(\xi)=0\,, \quad f(\xi)=\sum_{k\geq1}f_k\xi^k\,.
\end{equation*}
The coefficients of $f_k$ are uniquely determined by the recursive relation 
\begin{equation}\label{recursive1}
    f_1=g_0, \quad f_n=\frac{1}{n}\sum_{m=1}^{n-1}g_m\sum_{k_1+...+k_m=n-1}f_{k_1}\cdot...\cdot f_{k_m}\,.
\end{equation}
In particular, note that all coefficients are non-negative. Since $G(u)$ is analytic, the formal series $f(\xi)$ actually converges in a neighbourhood of $\xi=0$, since there is a unique analytic solution near $\xi=0$ vanishing at $\xi=0$, and the coefficients of the expansion of the analytic solution must satisfy \eqref{recursive1}.\\ 

We claim that $C(\xi)\leq f(\xi).$ First, notice that by \eqref{in5}
\begin{equation*}
    C'(0)=c_1\leq [\xi^0]G(C(\xi))=g_0, \quad f_1=f'(0)=G(0)=g_0
\end{equation*}
so $c_1\leq f_1.$ On the other hand, again by \eqref{in5}
\begin{equation*}
    2c_2=[\xi]C'(\xi)\leq [\xi]G(C(\xi))=g_1c_1\leq g_1f_1=[\xi]G(f(\xi))=[\xi]f'(\xi)=2f_2 \implies c_2\leq f_2\,.
\end{equation*}
Now assume that we have proved the result by induction up to $N$. Then by \eqref{in5}
\begin{equation*}
\begin{split}
    (N+1)c_{N+1}&=[\xi^{N}]C'(\xi)\leq [\xi^N]G(C(\xi))=\sum_{m=1}^{N}g_m\sum_{k_1+...+k_m=N}c_{k_1}\cdot...\cdot c_{k_m}\\
    &\leq \sum_{m=1}^{N}g_m\sum_{k_1+...+k_m=N}f_{k_1}\cdot...\cdot f_{k_m}=[\xi^N]G(f(\xi))=[\xi^N]f'(\xi)=(N+1)f_{N+1}
\end{split}
\end{equation*}
so $c_{N+1}\leq f_{N+1}. $ It then follows that $C(\xi)\leq f(\xi)$, and since $f(\xi)$ converges in a neighbourhood of $\xi=0$, the $C(\xi)$ must also converge in a neighbourhood of $\xi=0$. This shows that $\mathcal{B}[\dot{g}](\xi)$ converges in a neighbourhood of $\xi=0$, uniformly in $(Z^i,\theta^i)\in B_r(p)$.\\

In order to prove the more general statement of the theorem, consider $\dot{g}$ with $\dot{g}|_M=0$ as before, but assume now that $p\in TM$ is any point not necessary on the zero section. Join $p$ and $\pi(p)\in M\subset TM$ with a path on the fiber $T_{\pi(p)}M$, and cover the path by finitely many polydiscs. The main thing that changes in the proof is that when we use \eqref{pathintest} when working on a polydisc, we have an extra boundary term. We can use induction and the previous proof as the base case to show that these boundary terms contribute to convergent expressions. The same argument then works for $\dot{g}$ not necessarily vanishing on $M\subset TM$, provided we assume that the Borel transform converges when restricted to the zero section. 
\end{proof}
\subsection{Borel transform of formal Darboux coordinates}

In Section \ref{btgaugesec} we have shown that for a Joyce structure over $M$ with associated relative connection $\mathcal{A}$, and any $\dot{g}\in \text{Lie}(\widehat{\mathcal{G}})$ satisfying $e^{\mathcal{L}_{\dot{g}}}\mathcal{A}^{\text{st}}=\mathcal{A}$, we have that $\dot{g}$ has a convergent Borel transform provided the same holds for $\dot{g}|_M$. In particular, for the $\dot{g}$ constructed in Theorem \ref{mt1} satisfying $\dot{g}|_M=0$ the Borel transform automatically converges. \\

On the other hand, in Section \ref{formaldarbouxcoordsec} we used $\dot{g}$ to construct formal Darboux coordinates for the formal twistor space by setting
\begin{equation*}
    x^i=e^{\dot{g}}(x^{i,\text{st}}), \quad x^{i,\text{st}}=\theta^i-Z^i/\epsilon,
\end{equation*}
where $(Z^i)$ are flat Darboux coordinates of $\Omega$ on some neighbourhood $U\subset M$, and $(Z^i,\theta^i)$ the induced coordinates on $\pi^{-1}(U)\subset TM$. Since $x^i-x^{i,\text{st}}\in \epsilon \cdot  \mathcal{O}_{TM}(\pi^{-1}(U))[[\epsilon]]$ and we have Theorem \ref{mt2}, it is natural to ask whether the Borel transform of $x^i-x^{i,\text{st}}$ converges in a neighbourhood of $\xi=0$.
\begin{theorem}\label{mt3}
    Let $x^i, x^{i,\text{st}}\in \epsilon^{-1}\cdot \mathcal{O}_{TM}(\pi^{-1}(U))[[\epsilon]]$ be as in Section \ref{formaldarbouxcoordsec} and let $\dot{g}$ be as in Theorem \ref{mt1}. Then $\mathcal{B}[x^i-x^{i,\text{st}}](\xi)$ converges in $\xi$ in a neighbourhood of $\xi=0$,  uniformly in $p\in B_r\subset TM$ where $B_r$ is a polydisc centered at the zero section $M\subset TM$.
\end{theorem}

\begin{proof}
    Note that since the coefficients of $\dot{g}$ are vertical vector fields, we can write
    \begin{equation*}
        x^i = e^{\dot{g}}(x^{i,\text{st}})=x^{i,\text{st}}+(e^{\dot{g}}-1)(\theta^i)\,,
    \end{equation*}
    where $(e^{\dot{g}}-1)(\theta^i)\in \epsilon \cdot  \mathcal{O}_{TM}(\pi^{-1}(U))[[\epsilon]]$.
    Hence it is enough to show that \begin{equation*}
        \mathcal{B}[(e^{\dot{g}}-1)(\theta^i)] = \sum_{k=1}^{\infty}\frac{\mathcal{B}[\dot{g}^k(\theta^i)]}{k!}
    \end{equation*}
    converges on a neighbourhood of $\xi=0$, uniformly in $p\in B_r$ where $B_r$ is a polydisc centered at the zero section $M\subset TM$. To deal with the Borel transform of several $\dot{g}$ applied to $\theta^i$, note that Borel transforms takes products to convolutions. Namely, if $A$ and $B$ are two differential operators given as series in $\epsilon$, then 
    \begin{equation}\label{convform}
        \mathcal{B}[AB](\xi)=\int_{0}^{\xi}\mathcal{B}[A](\eta)\mathcal{B}[B](\xi-\eta)\mathrm{d}\eta\,= (\mathcal{B}[A]*\mathcal{B}[B])(\xi)\,.
    \end{equation}
    In terms of the expansions in $\xi$, the convolution product $*$ satisfies 
    \begin{equation*}
        \frac{\xi^{k-1}}{(k-1)!}* \frac{\xi^{l-1}}{(l-1)!}=\frac{\xi^{k+l-1}}{(k+l-1)!}\,.
    \end{equation*}
    We can then write
    \begin{equation*}
        \mathcal{B}[(e^{\dot{g}}-1)(\theta^i)]=\sum_{k=1}^{\infty}\frac{\mathcal{B}[\dot{g}]^{*k}(\theta^i)}{k!}\,.
    \end{equation*}
    Now let $p\in M\subset TM$. By Theorem \ref{mt2} we can find $0<R$ sufficiently small so that $\mathcal{B}[\dot{g}](\xi)$ converges for $|\xi|<R$ uniformly in $(Z^i,\theta^i)\in B_R(p)\subset \pi^{-1}(U)\subset TM$. 
    In the following we use the following for vector fields $V=V^j(Z^i,\theta^i,\xi)\partial_{\theta^j}$, and functions $f(Z^i,\theta^i,\xi)$ holomorphic on a neighbourhood of $K\subset TM\times \mathbb{C}$:
    \begin{equation*}
        ||f||_{K}=\text{sup}_{(Z^i,\theta^i,\xi)\in K}|f(Z^i,\theta^i,\xi)|, \quad ||V||_K=\max_{j=1,...,n}\text{sup}_{(Z^i,\theta^i,\xi)\in K}|V^j(Z^i,\theta^i,\xi)|\,.
    \end{equation*}
    In particular, by possibly shrinking $R$ and denoting $\widetilde{B}_R:=B_R(p)\times \{\xi \in \mathbb{C} \; | \; |\xi|<R\}$ we have
    \begin{equation*}
        M_R := ||\mathcal{B}[\dot{g}]||_{\widetilde{B}_R}<\infty\,.
    \end{equation*}
    We now prove the following lemma, which will be our main estimate.

    \begin{lemma}
        For $k\geq 2$, $0<r<R$, and $\rho_k>0$ such that $r+(k-1)\rho_k\leq R$ we have  
        \begin{equation}\label{est1}
            ||\mathcal{B}[\dot{g}]^{*k}(\theta^i)||_{\widetilde{B}_{r}}\leq M_{r+(k-1)\rho_k}^k \left(\frac{nR}{\rho_k}\right)^{k-1}\,.
        \end{equation}
    \end{lemma}
    \begin{proof}
        We prove this by induction. For $k=2$ we have using \eqref{convform} and the parametrization $\eta=t\cdot \xi$ that for $|\xi|<r$
        \begin{equation*}
            |\mathcal{B}[\dot{g}]^{*2}(\xi)(\theta^i)|\leq \int_0^1|\mathcal{B}[\dot{g}](t\xi)\mathcal{B}[\dot{g}]((1-t)\xi)(\theta^i)|\cdot |\xi|\mathrm{d}t
        \end{equation*}
    Using that $\mathcal{B}[\dot{g}](\theta^i)=\mathcal{B}[\dot{g}]^{j}\partial_{\theta^{j}}(\theta^i)=\mathcal{B}[\dot{g}]^i$ and the Cauchy estimate as in the proof of Lemma \ref{prelem} to we find that 
    \begin{equation*}
        \int_0^1|\mathcal{B}[\dot{g}](t\xi)\mathcal{B}[\dot{g}]((1-t)\xi)(\theta^i)|\cdot |\xi|\mathrm{d}t\leq ||\mathcal{B}[\dot{g}]||_{\widetilde{B}_{r+\rho_2}}\cdot||\mathcal{B}[\dot{g}]^i||_{\widetilde{B}_{r+\rho_2}}\cdot \frac{nR}{\rho_2}\leq M_{r+\rho_2}^2\cdot \frac{nR}{\rho_2}\,.
    \end{equation*}
    so that the case $k=2$ holds. \\

    Now assume that it holds for $l \geq 2$ and let us show that it holds for $l+1$. Since $\rho_{l+1}$ is chosen such that $r+l\rho_{l+1}\leq R$, then in particular $r+\rho_{l+1}\leq R$. By applying the same argument as before we find 
    \begin{equation*}
        ||\mathcal{B}[\dot{g}]^{*l+1}(\theta^i)||_{\widetilde{B}_r}\leq M_{r+\rho_{l+1}}\cdot  ||\mathcal{B}[\dot{g}]^{*l}(\theta^i)||_{\widetilde{B}_{r+\rho_{l+1}}}\frac{nR}{\rho_{l+1}}\,.
    \end{equation*}
    On the other hand, we have that $(r+\rho_{l+1})+(l-1)\rho_{l+1}\leq R$, so applying the inductive hypothesis on $||\mathcal{B}[\dot{g}]^{*l}(\theta^i)||_{\widetilde{B}_{r+\rho_{l+1}}}$ we find
    \begin{equation*}
        ||\mathcal{B}[\dot{g}]^{*l+1}(\theta^i)||_{\widetilde{B}_r}\leq M_{r+\rho_{l+1}}M_{r+\rho_{l+1}+(l-1)\rho_{l+1}}^l\left(\frac{nR}{\rho_{l+1}}\right)^{l}\leq M_{r+l\rho_{l+1}}^{l+1}\left(\frac{nR}{\rho_{l+1}}\right)^{l}\,,
    \end{equation*}
    which shows what we want. 
    \end{proof}
    Using the estimate \eqref{est1} from the previous lemma with $\rho_k=\frac{R-r}{k-1}$ we find that for all $k\geq 2$
    \begin{equation*}
        ||\mathcal{B}[\dot{g}]^{*k}(\theta^i)||_{\widetilde{B}_{r}}\leq M_{R}^k \left(\frac{nR(k-1)}{R-r}\right)^{k-1}\,.
    \end{equation*}
    Hence, we can write
    \begin{equation*}
        ||\mathcal{B}[e^{\dot{g}}-1](\theta^i)||_{\widetilde{B}_r}\leq M_R+ \sum_{k=2}^{\infty}\frac{M_{R}^k }{k!}\left(\frac{nR(k-1)}{R-r}\right)^{k-1}
    \end{equation*}
    Using the quotient convergence test we see that the above series converges provided that 
    \begin{equation}\label{est2}
        M_R\frac{nRe}{R-r}<1\,.
    \end{equation}
    Since $\mathcal{B}[\dot{g}]|_M=0$ we can make $M_R$ as small as we want by taking $R$ sufficiently small, and hence we can pick $R$ such that \eqref{est2} holds. This shows that $\mathcal{B}[e^{\dot{g}}-1](\theta^i)$ converges uniformly in $(Z^i,\theta^i,\xi)\in \widetilde{B}_r=B_r(p)\times \{\xi \in \mathbb{C} \; | \; |\xi|<r\}$.
\end{proof}
\begin{remark}
    Note that in the above proof we only used the assumption that $\dot{g}|_M=0$ to show that, by choosing $R$ sufficiently small, we could make $M_R$ small enough for \eqref{est2} hold, and hence guarantee that the Borel transform converges. For more general polydiscs $\widetilde{B}_R=B_R(p)\times \{|\xi|<R\}$ where $p$ is not necessarily on the zero section, note that 
    \begin{equation}\label{est3}
        M_R=||\mathcal{B}[\dot{g}]||_{\widetilde{B}_R}\leq ||\dot{g}_1||_{\widetilde{B}_R}+||\mathcal{B}[\dot{g}]-\dot{g}_1||_{\widetilde{B}_R}\,.
    \end{equation}
    The second summand in \eqref{est3} can be made as small as desired by simply choosing $|\xi|$ sufficiently small. On the other hand, while $\dot{g}_1$ is $\xi$-independent, using that $\dot{g}_1=-\text{Ham}^{\Omega^{v}}(W)$ and the fact that (see \cite[Equation (45)]{TwistorJoyce} 
    \begin{equation*}
        W(\lambda Z^i,\theta^i)=\lambda^{-1}W(Z^i,\theta^i)
    \end{equation*}
    suggests that  $||\dot{g}_1||_{B_R}$ can be made small by taking the center $p$ of the polydisc $B_R(p)$ such that $\pi(p)\in M$ goes towards one of the infinite ends of $M$. 
\end{remark}
\subsection{Towards endless analytic continuation}\label{endlesssec}

In this section we make some comments on what could be done to study the endless analytic continuation property of the Borel transform of $\dot{g}$, and hence show that $\dot{g}$ is resurgent. This section in mostly speculative, and we make no claim that the proposed path is the most optimal way to show this. \\

A natural way to study the analytic continuation properties of $\mathcal{B}[\dot{g}](\xi)$  as a function of $\xi$ is to first find a differential equation in $\xi$ obeyed by $\mathcal{B}[\dot{g}](\xi)$. One can then try showing via contraction mapping techniques uniqueness and existence of solutions, and study singularities of the solutions via the structure of the differential equation. These two could in turn be combined to establish the endless analytic continuation property of $\mathcal{B}[\dot{g}](\xi)$. See for example \cite{alameddine2025resurgencetritronqueessolutionsdeformed} for a recent application of this that might be relevant for the A2-quiver Joyce structure presented in the examples. \\

To come up with a differential equation in $\xi$ obeyed by $\mathcal{B}[\dot{g}](\xi)$, we first find a formal differential equation in $\epsilon$ obeyed by $\dot{g}$, and then take its Borel transform.\\

Throughout this subsection, we fix a Joyce structure $(M,\Omega,\Gamma,Z,h)$ with corresponding flat meromorphic relative connection $\mathcal{A}$ on $\hat{p}:TM\times \widehat{D}\to M\times \widehat{D}$. We further assume that we have $\dot{g}\in \text{Lie}(\widehat{\mathcal{G}})$ solving $e^{\mathcal{L}_{\dot{g}}}\mathcal{A}^{\text{st}}=\mathcal{A}$ and satisfying the homogeneity condition 
\begin{equation}\label{homcond}
    [E,\dot{g}_k]=-k\dot{g}_k\,,
\end{equation}
where $E=\mathcal{H}_{Z}$. An instance of such a solution is given by the one constructed in Theorem \ref{mt1}. To see that this $\dot{g}$ satisfies the homogeneity condition, note that since the Pleba\'nski function $W$ satisfies $EW=-W$ (see \cite[Lemma 3.7]{TwistorJoyce}), and  locally $\dot{g}_1=-\text{Ham}^{\Omega^{v}}(W)$ (see Remark \ref{g1pleb}), then we automatically get $[E,\dot{g}_1]=-\dot{g}_1$. Equations \eqref{receq} then imply that 
\begin{equation*}
    [v_i,[E,\dot{g}_k]]=-k[v_i,\dot{g}_k],
\end{equation*}
so integrating along the fibers and using that both $[E,\dot{g}_k]$ and $\dot{g}_k$ vanish along the zero section $M\subset TM$ we obtain the required homogeneity condition \eqref{homcond}. Note that any $\dot{g}$ satisfying \eqref{homcond} condition also satisfies
\begin{equation}\label{fulldotghom}
    [\epsilon\partial_{\epsilon}+E, \dot{g}]
=0\,.\end{equation}
On the other hand, as mentioned in Remark \ref{epsilonext}, we can extend $\mathcal{A}$ to a full meromorphic flat connection connection on $\hat{p}$ by defining

\begin{equation*}
    \mathcal{A}_{\epsilon \partial_{\epsilon}}=\epsilon\partial_{\epsilon}-\frac{1}{\epsilon}v_Z - \omega_Z\,.
\end{equation*}
Using that $e^{\mathcal{L}_{\dot{g}}}\mathcal{A}^{\text{st}}_X=\mathcal{A}_X$ for all vector fields on $M$, the homogeneity condition \eqref{fulldotghom}, and denoting $\mathcal{A}^{\text{st}}_{\epsilon \partial_{\epsilon}}=\epsilon\partial_{\epsilon}-\frac{1}{\epsilon}v_Z$, we find that

\begin{equation*}
     e^{\mathcal{L}_{\dot{g}}}(\mathcal{A}^{\text{st}}_{\epsilon \partial_{\epsilon}})= e^{\mathcal{L}_{\dot{g}}}(\epsilon\partial_{\epsilon}+E - \mathcal{A}^{\text{st}}_Z)=\epsilon \partial_{\epsilon}+E - \mathcal{A}_Z=\mathcal{A}_{\epsilon \partial_{\epsilon}}\,.
\end{equation*}
Hence, $\dot{g}$ is such that $e^{\mathcal{L}_{\dot{g}}}\mathcal{A}^{\text{st}}=\mathcal{A}$ holds for the extended $\mathcal{A}^{\text{st}}$ and $\mathcal{A}$.\\

We can rewrite the condition $e^{\mathcal{L}_{\dot{g}}}\mathcal{A}^{\text{st}}_{\epsilon\partial_{\epsilon}}=\mathcal{A}_{\epsilon\partial_{\epsilon}}$ as
\begin{equation*}
    \sum_{k\geq 1}\frac{\mathcal{L}_{\dot{g}}^k}{k!}(\epsilon\partial_{\epsilon}-\frac{1}{\epsilon}v_Z)=-\omega_Z\,,
\end{equation*}
which in turn can be written as 
\begin{equation}\label{diffepsilon}
    \epsilon\sum_{k\geq 1}\frac{\mathcal{L}_{\dot{g}}^{k-1}}{k!}(\partial_{\epsilon}\dot{g})+\frac{1}{\epsilon}\sum_{k\geq 1}\frac{\mathcal{L}_{\dot{g}}^k(v_Z)}{k!}-\omega_Z=0\,.
\end{equation}
We can formally invert $\sum_{k\geq 1}\frac{\mathcal{L}_{\dot{g}}^{k-1}}{k!}$ by multiplying \eqref{diffepsilon} with the formal series
\begin{equation*}
    \psi_{\dot{g}}:=\frac{\mathcal{L}_{\dot{g}}}{e^{\mathcal{L}_{\dot{g}}}-1}=\sum_{k\geq 0}\frac{B_k}{k!}\mathcal{L}_{\dot{g}}^k
\end{equation*}
where $B_k$ are Bernoulli numbers. We obtain the following formal differential equation of $\dot{g}$ in the $\epsilon$-variable
\begin{equation}\label{finalepsilondiffeq}
    \epsilon\partial_{\epsilon}\dot{g}+\frac{1}{\epsilon}\mathcal{L}_{\dot{g}}(v_Z)-\psi_{\dot{g}}(\omega_Z)=0\,.
\end{equation}
Applying the Borel transform for \eqref{finalepsilondiffeq}, we can obtain a differential equation in $\xi$ for $\mathcal{B}[\dot{g}](\xi)$ by using the formulas
\begin{equation*}
    \mathcal{B}[\epsilon \partial_\epsilon A]=\mathcal{B}[A]+\xi \partial_{\xi}\mathcal{B}[A], \quad \mathcal{B}[AB]=\mathcal{B}[A]*\mathcal{B}[B], \quad A,B\in \mathbb{C}[[\epsilon]]
\end{equation*}
where $*$ denotes the convolution product satisfying 
\begin{equation*}
\frac{\xi^{k-1}}{(k-1)!}*\frac{\xi^{l-1}}{(l-1)!}=\frac{\xi^{k+l-1}}{(k+l-1)!}\,.
\end{equation*}
In principle, studying the solutions and singularity structure of solutions of the resulting differential equation could illuminate the issue of analytic continuation of $\mathcal{B}[\dot{g}]$, but the resulting equation seems rather complicated. For example, it is hard to make global sense of the Borel transform of the term $\psi_{\dot{g}}(\omega_Z)$. We hope to tackle some of these issues in future work.

\section{Examples}\label{exsec}

In this section we present two examples associated to the DT theory of the A1 and A2 quiver. The A1 case is rather simple and everything can be done explicitly. For this example we will construct two infinitesimal gauge transformations gauging $\mathcal{A}^{\text{st}}$ to $\mathcal{A}$ that have wildly different resurgent behaviour.  For the A2 case, we discuss the Joyce structure and explicitly compute the first two terms $\dot{g}_1$ and $\dot{g}_2$ of a $\dot{g}$ gauging $\mathcal{A}^{\text{st}}$ to $\mathcal{A}$, which could be used to compute approximations of twistor Darboux coordinates and tau-functions \cite{BridgelandOsc, BridgelandTau}. We remark that the A2-quiver Joyce structure is related to the isomonodromy flows of the Painlev\'e I equation. Finally, we note that the A1 case is easily extended to the case that the DT invariants are finite and uncoupled by the usual superposition of several A1 cases. 
\subsection{A1 quiver}
We take $M=\mathbb{C}^{*}\times \mathbb{C}$ with global coordinates $(z,\check{z})$, so that $TM$ has induced global coordinates $(z,\check{z},\theta,\check{\theta})$. $M$ carries the holomorphic symplectic form
\begin{equation*}
    \Omega=\frac{1}{2\pi \mathrm{i}}\mathrm{d}z\wedge \mathrm{d}\check{z}\,.
\end{equation*}
The period structure $(\Gamma,Z)$ is given by the trivial bundle of lattices $\Gamma\subset TM$ where $\Gamma|_{p}=\text{span}_{\mathbb{Z}}(\partial_{z}|_p,\partial_{\check{z}}|_p)$ and the vector field $Z=z\partial_z+\check{z}\partial_{\check{z}}$.  In particular, $(z,\check{z})$ are flat Darboux coordinates for $\Omega$ with respect to the linear flat connection $\nabla$ determined by $\Gamma$. The connection $h:\pi^*(TM)\to T(TM)$ is given in the global coordinates $(z,\check{z},\theta,\check{\theta})$ of $TM$ by
\begin{equation*}
    h_{\partial_z}=\partial_z+\frac{\theta}{2\pi \mathrm{i}z}\partial_{\check{\theta}}, \quad h_{\partial_{\check{z}}}=\partial_{\check{z}}\,.
\end{equation*}
The $\mathbb{C}^{*}$-family of flat symplectic connections $\mathcal{A}^{\epsilon}$ is then given by 
\begin{equation}
    \mathcal{A}_{\partial_{z}}^{\epsilon}=\partial_{z}+\frac{1}{\epsilon}\partial_{\theta} +\frac{\theta}{2\pi \mathrm{i}z}\partial_{\check{\theta}}, \quad \mathcal{A}_{\partial_{\check{z}}}^{\epsilon}=\partial_{\check{z}}+\frac{1}{\epsilon}\partial_{\check{\theta}}\,\,.
\end{equation}
It is easy to check that $(M,\Omega,\Gamma,Z,h)$ as above satisfies all the properties of a Joyce structure except the periodicity along the fibres (see (J2) of Definition \ref{Joyce_def}). Nevertheless, this property is never used in the proofs of Theorems \ref{mt1}, \ref{mt2} and \ref{mt3}, so our results still apply. For more details of how this pre-Joyce structure is obtained from a solution of a Riemann-Hilbert problem associated to the DT-invariants of the A1 quiver, see \cite{BridgeJoyce, varBPS}. \\

For this simple case there is a global Pleba\'nski function satisfying the uniqueness conditions \eqref{plebzerosec} given by
\begin{equation}\label{A1pleb}
    W=-\frac{\theta^3}{6(2\pi \mathrm{i})^2z}\,.
\end{equation}
Indeed, we have
\begin{equation*}
\begin{split}
    \text{Ham}^{\Omega^v}(v_{\partial_z}W)&=\text{Ham}^{\Omega^v}(\partial_{\theta}W)=\Omega^{z\check{z}}\partial^2_{\theta}W\partial_{\check{\theta}}=(-2\pi \mathrm{i})\cdot \left(-\frac{\theta}{(2\pi \mathrm{i})^2z}\right)\partial_{\check{\theta}}=\frac{\theta}{2\pi \mathrm{i}z}\partial_{\check{\theta}}\,\\
    \text{Ham}^{\Omega^v}(v_{\partial_{\check{z}}}W)&=\text{Ham}^{\Omega^v}(\partial_{\check{\theta}}W)=0\,.\\
\end{split}
\end{equation*}
We will now describe two  infinitesimal gauge transformations $\dot{g}\in \text{Lie}(\widehat{\mathcal{G}})$ satisfying
\begin{equation*}
e^{\mathcal{L}_{\dot{g}}}\mathcal{A}^{\text{st}}=\mathcal{A}.
\end{equation*}
The first one is the one from Theorem \ref{mt1} satisfying $\dot{g}|_M=0$, while the second one will be related to the solution of the Riemann-Hilbert problem associated to DT-invariants formulated in \cite{varBPS}.
\subsubsection{Infinitesimal gauge transformation vanishing on the zero-section}
Let us first explicitly construct $\dot{g}$ from Theorem \ref{mt1}. If we denote $(Z^1,Z^2,\theta^1,\theta^2)=(z,\check{z},\theta,\check{\theta})$, then by Proposition Proposition \ref{geq} we want to solve \begin{equation}\label{A1eq1}
    [\dot{g}_1,\partial_{\theta^i}]=\text{Ham}^{\Omega^{v}}(\partial_{\theta^i}W), \quad i=1,2,
\end{equation}
and for each $l>0$ and  $i=1,2$

\begin{equation}\label{A1eq2}
0=\sum_{r=1}^{l+1}\frac{1}{r!}\sum_{k_1+...+k_r=l+1}[\dot{g}_{k_1},...,[\dot{g}_{k_r},\partial_{\theta^i}]...] + \sum_{r=1}^{l}\frac{1}{r!}\sum_{k_1+...+k_r=l}[\dot{g}_{k_1},...,[\dot{g}_{k_r},\partial_{Z^i}],...]\,,    
\end{equation}
while at the same time ensuring the boundary condition $\dot{g}|_M=0$. We look for a solution of the form 
\begin{equation}\label{a1ansatz}
    \dot{g}=\sum_{k=1}^{\infty}a_k(z,\theta)\epsilon^{k}\partial_{\hat{\theta}}\,,
\end{equation}
where $a_k(z,\theta)$ are complex-valued functions of $(z,\theta)$. It is not hard to check that this form automatically satisfies \eqref{A1eq1} and \eqref{A1eq2} for $i=2$ since $W$ and the coefficients of $\dot{g}$ in \eqref{a1ansatz} depend only on $(z,\theta)$. Furthermore, the equations \eqref{A1eq1} and \eqref{A1eq2} for $i=1$ reduce to 
\begin{equation}\label{a1eq}
    \partial_{\theta}a_1=-\Omega^{z\hat{z}}\frac{\partial^2 W}{\partial \theta^2}=2\pi \mathrm{i}\frac{\partial^2 W}{\partial \theta^2}, \quad \partial_{\theta}a_{k+1}=-\partial_{z}a_k \quad \text{for} \quad k>0\,.
\end{equation}
It is easy to check that a solution of \eqref{a1eq} satisfying $a_k(z,0)=0$ is given by 
\begin{equation*}
   a_k(z,\theta)=-\frac{1}{2\pi \mathrm{i}}\frac{(k-1)!\theta^{k+1}}{(k+1)!z^k}\,.
\end{equation*}
The resulting infinitesimal gauge transformation 
\begin{equation}\label{a1firstsol}
    \dot{g}=-\frac{1}{2\pi \mathrm{i}}\left(\sum_{k=1}^{\infty}\frac{(k-1)!\theta^{k+1}}{(k+1)!z^k}\epsilon^k\right)\partial_{\hat{\theta}}
\end{equation}
clearly vanishes along the zero section of $TM\to M$, and hence describes the unique solution from Theorem \ref{mt1}. Note that infinite sum actually converges for $|\epsilon\theta/z|<1$ and admits the analytic continuation given by 
\begin{equation*}
    \dot{g} = -\frac{1}{2\pi \mathrm{i}}\left(\theta+\frac{z-\epsilon \theta}{\epsilon}\log\left(1-\frac{\epsilon \theta}{z}\right)\right)\partial_{\hat{\theta}}\,.
\end{equation*}
The fact that \eqref{a1firstsol} converges  implies that the Borel transform $\mathcal{B}[\dot{g}]$ has an infinite radius of convergence, and hence no interesting singularities.  \\

Note that in this case the formal Darboux coordinates are given on $TM\times \mathbb{C}^*$ by
\begin{equation*}
    x^1=e^{\dot{g}}x^{1,\text{st}}=x^{1,\text{st}}=\theta-\frac{z}{\epsilon}, \quad x^2=e^{\dot{g}}x^{2,\text{st}}=\check{\theta}-\frac{\check{z}}{\epsilon}-\frac{1}{2\pi \mathrm{i}}\left(\theta+\frac{z-\epsilon \theta}{\epsilon}\log\left(1-\frac{\epsilon \theta}{z}\right)\right)
\end{equation*}
and for fixed $\epsilon\in \mathbb{C}^{*}$ we have
\begin{equation}\label{hktwistora1}
    \Omega_{z\check{z}}\mathrm{d}x^1\wedge \mathrm{d}x^2= \frac{1}{2\pi \mathrm{i}\epsilon^2}\mathrm{dz}\wedge \mathrm{d}\check{z}-\frac{1}{2\pi \mathrm{i}\epsilon}\left(\mathrm{d}\theta \wedge \mathrm{d}\check{z}+\mathrm{d}z\wedge \mathrm{d}\check{\theta}\right)+\frac{1}{2\pi\mathrm{i}}\mathrm{d}\theta\wedge \mathrm{d}\check{\theta} -\frac{\theta}{(2\pi \mathrm{i})^2z}\mathrm{d}\theta \wedge \mathrm{d}z\,.
\end{equation}
In particular, from \eqref{hktwistora1} and \cite[Equation (24-26)]{TwistorJoyce} we see that if $q:TM\times \mathbb{P}^1\to \mathcal{Z}$ is the projection into the twistor space and $\varpi$ is the $\mathcal{O}(2)$-twisted relative symplectic form on $\mathcal{Z}$ then  
\begin{equation}\label{a1pullbacksymp}
    q^*\varpi|_{TM\times \{\epsilon\}} = \epsilon^2\Omega_{z\check{z}}\mathrm{d}x^1\wedge \mathrm{d}x^2 \otimes \partial_{\epsilon}\,.
\end{equation}
By Proposition \ref{flatcoordsprop} and \eqref{a1pullbacksymp} we conclude that $x^1$ and $x^2$ descend to twistor Darboux coordinates for $\varpi$. We remark that these are true analytic coordinates rather than formal Darboux coordinates, since $\dot{g}$ converges. 

\begin{remark}
Note that $(\epsilon x^1,\epsilon x^2)$ also descend to twistor Darboux coordinates, and they are well-defined near $\epsilon=0$. This is in contrast with the twistor coordinates obtained via solutions of the Riemann-Hilbert problem associated to DT-invariants \cite{varBPS}, which are usually defined for $\epsilon$ in a sector in $\mathbb{C}^*$.
\end{remark}

\subsubsection{Infinitesimal gauge transformation related to the Riemann-Hilbert problem}

We now construct a different $\dot{g}$ related to the Riemann-Hilbert problem studied in \cite{varBPS}. To do this, we first modify the global Pleba\'nski function  \eqref{A1pleb} by adding a term linear in $\theta$. We take
\begin{equation}
    \widetilde{W}=\frac{2\pi \mathrm{i}}{6z}B_3\left(\frac{\pi \mathrm{i}-\theta}{2\pi \mathrm{i}}\right)=-\frac{\theta^3}{6(2\pi \mathrm{i})^2z}+\frac{\theta}{24z}= W+\frac{\theta}{24z}
\end{equation}
where $B_n(w)$ denotes the n-th Bernoulli polynomial. While the addition of $\frac{\theta}{24z}$ seems ad-hoc at this moment, it is worthwhile to note that
\begin{equation*}
    \partial_{\theta}\widetilde{W}|_{M}=\partial_zS,\quad \text{where}\quad S(z)=\log(z^{1/24})\,
\end{equation*}
and compare with \cite[Equations (21) and (22)]{BridgeFab}. Note that 
\begin{equation*}
    \text{Ham}^{\Omega^{v}}(v_X\widetilde{W})=\text{Ham}^{\Omega^{v}}(v_XW)
\end{equation*}
so we can use $\widetilde{W}$ as another Pleba\'nski function for the A1 pre-Joyce structure, but $\widetilde{W}$ does not satisfy \eqref{plebzerosec}. \\

As before, we will find $\dot{g}$ satisfying \eqref{A1eq1} and \eqref{A1eq2}, but now $\dot{g}|_M\neq 0$.  We  prove that a solution is given by 
\begin{equation}\label{a1altsol}
\dot{g}=\sum_{k=1}^{\infty}\epsilon^k\text{Ham}^{\Omega^{v}}(W_k(z,\theta)), \quad W_k=\left(\frac{2\pi \mathrm{i}}{z}\right)^kB_{k+2}\left(\frac{\pi \mathrm{i}-\theta}{2\pi \mathrm{i}}\right)\frac{(-1)^{k}(k-1)!}{(k+2)!} \quad \text{for} \quad k\geq 1\,,
\end{equation}
where 
\begin{equation}\label{a1althams}
    \text{Ham}^{\Omega^{v}}(W_k(z,\theta))=-2\pi \mathrm{i}\frac{\partial W_k(z,\theta)}{\partial \theta}\partial_{\check{\theta}}\,.
\end{equation}  
As before, the fact that $\dot{g}$ is independent of the $(\check{z},\check{\theta})$ variables means that we only need to check
\begin{equation}\label{a1eqalt}
    \partial_{\theta}\tilde{a}_1=-\Omega^{z\hat{z}}\frac{\partial^2 W}{\partial \theta^2}=2\pi \mathrm{i}\frac{\partial^2 W}{\partial \theta^2}, \quad \partial_{\theta}\tilde{a}_{k+1}=-\partial_{z}\tilde{a}_k \quad \text{for} \quad k>0\,.
\end{equation}
where 
\begin{equation*}
    \tilde{a}_k(z,\theta):=-2\pi \mathrm{i}\frac{\partial W_k(z,\theta)}{\partial \theta}\,.
\end{equation*}
The equation for $\tilde{a}_1$ follows from the fact that $W_1=-\widetilde{W}$ and the fact that $\partial^2_\theta W= \partial^2_{\theta^2}\widetilde{W}$, while the other equations in \eqref{a1eqalt} follow from the relation 
\begin{equation}\label{berrel}
    B_{n+1}'(w)=(n+1)B_{n}(w)\,.
\end{equation}
The required equations for $\dot{g}$ from Proposition \ref{geq} then hold. Note that this $\dot{g}$ satisfies $\dot{g}|_M\neq 0$.\\

Using \eqref{berrel} and replacing \eqref{a1althams} into \eqref{a1altsol} we can write explicitly 
\begin{equation}\label{a1altsolexplicit}
\dot{g}=\left(\sum_{k=1}^{\infty}\left(\frac{2\pi \mathrm{i}\epsilon}{z}\right)^kB_{k+1}\left(\frac{\pi \mathrm{i}-\theta}{2\pi \mathrm{i}}\right)\frac{(-1)^{k}}{(k+1)k}\right)\partial_{\check{\theta}}\,.
\end{equation}
Note that compared to \eqref{a1firstsol}, the formal series \eqref{a1altsolexplicit} is divergent due to the factorial divergence of the Bernoulli polynomials with respect to $k$. The Borel transform 
\begin{equation*}
    \mathcal{B}[\dot{g}](\xi)=\left(\sum_{k=1}^{\infty}\left(\frac{2\pi \mathrm{i}}{z}\right)^kB_{k+1}\left(\frac{\pi \mathrm{i}-\theta}{2\pi \mathrm{i}}\right)\frac{(-1)^{k}}{(k+1)!}\xi^{k-1}\right)\partial_{\check{\theta}}\,
\end{equation*}
has radius of convergence in $\xi$ equal to $|z|$, and using the generating function for Bernoulli polynomials we see that it admits the analytic continuation
\begin{equation}\label{a1analyticcont}
    \mathcal{B}[\dot{g}](\xi)=\frac{1}{\xi}\left(\frac{z}{2\pi \mathrm{i}\xi}+\frac{\theta}{2\pi \mathrm{i}}+\frac{\exp(\xi(\pi i + \theta)/z)}{1-\exp(2\pi i \xi/z) }\right)\partial_{\check{\theta}},
\end{equation}
with simple poles along $nz$ with $n\in \mathbb{Z} -\{0\}$. The location of the poles and the analytic continuation \eqref{a1analyticcont} of $\mathcal{B}[\dot{g}]$ shows that $\dot{g}$ is resurgent. \\

\begin{remark}
    \begin{itemize}
        \item The fact that the radius of convergence of $\mathcal{B}[\dot{g}](\xi)$ is independent of $\theta$ is a reflection of the statement of Theorem \ref{mt2}, where the convergence of $\mathcal{B}[\dot{g}]$ depends on the convergence of $\mathcal{B}[\dot{g}]|_M$.
        \item One can check that the Stokes automorphisms associated to the resurgent $\dot{g}$ are given by the expressions
        \begin{equation}\label{A1Stokes}
        \begin{split}
            \log(1+e^{\theta-z/\epsilon})&, \quad \text{along} \quad \mathbb{R}_{>0}\cdot z\\
            -\log(1+e^{z/\epsilon-\theta})&, \quad \text{along} \quad \mathbb{R}_{>0}\cdot (-z)\,.\\
        \end{split}
        \end{equation}
        These reproduce the logarithms of the jumps of the Riemann-Hilbert problem of the DT-invariants of the A1 quiver discussed in \cite{varBPS}, after the appropriate choice of quadratic refinement. 
    \end{itemize}
\end{remark}

We now use $\dot{g}$ to compute the formal twistor Darboux coordinates from Section \ref{formaldarbouxcoordsec}. We obtain 
\begin{equation}\label{a1formalDarboux}
    x^1=e^{\dot{g}}x^{1,\text{st}}=\theta-\frac{z}{\epsilon}, \quad x^2=e^{\dot{g}}x^{2,\text{st}}=\check{\theta}-\frac{\check{z}}{\epsilon}+\left(\sum_{k=1}^{\infty}\left(\frac{2\pi \mathrm{i}\epsilon}{z}\right)^kB_{k+1}\left(\frac{\pi \mathrm{i}-\theta}{2\pi \mathrm{i}}\right)\frac{(-1)^{k}}{(k+1)k}\right)\,.
\end{equation}
The formulas in \eqref{a1formalDarboux} match the asymptotic expansion as $\epsilon\to 0$ of the logarithm of the solution of the Riemann-Hilbert problem associated to the DT invariants of the A1 quiver (see \cite[Section 8.2]{BridgeJoyce}). Since $x^2-x^{2,\text{st}}$ is the $\partial_{\check{\theta}}$ component of $\dot{g}$, it follows from the previous arguments that $x^2-x^{2,\text{st}}$ is resurgent  and its associated Stokes automorphisms are the expected ones from \eqref{A1Stokes}. \\

Finally, note that even though the expression for $x^2$ is a formal expression in $\epsilon$ with infinitely many non-trivial terms, the combination $\Omega_{z\check{z}}\mathrm{d}x^1\wedge \mathrm{d}x^2$ still satisfies \eqref{hktwistora1}, with all contributions to $\epsilon^k$-order with $k\geq 1$ vanishing.\\

We summarize all the statements from this subsection in the following proposition.

\begin{proposition} For the A1 pre-Joyce structure, $\dot{g}\in \text{Lie}(\widehat{\mathcal{G}})$ given by \eqref{a1altsolexplicit} satisfies $e^{\mathcal{L}_{\dot{g}}}\mathcal{A}^{\text{st}}=\mathcal{A}$ and is resurgent. The corresponding formal Darboux coordinates $x^i=e^{\dot{g}}x^{i,\text{st}}$
satisfy that $x^i-x^{i,\text{st}}$ is resurgent, and the Stokes automorphisms match the logarithm of the jumps of the Riemann-Hilbert problem associated to the DT invariants of the A1  quiver. 
    
\end{proposition}

\subsection{A2 quiver}

Let's start by describing the tuple $(M,\Omega,\Gamma,Z,h)$ of the Joyce structure. For more details, see \cite{BridgelandOsc} and \cite[Section  9]{TwistorJoyce}. \\

The base $M$ of the Joyce structure is given by

\begin{equation*}
    M=\{(a,b)\in \mathbb{C}^2 \; | \; 4a^3+27b^2\neq 0\}\,.
\end{equation*}
A point in $M$
 should be thought as defining a quadratic differential $(x^3+ax+b)(\mathrm{d}x)^2$ on $\mathbb{P}^1$ with simple zeroes and an order $7$ pole at $x=\infty$. In order to define $\Gamma \to M$, first consider the affine elliptic curve associated to a quadratic differential on $M$:
 \begin{equation}
     \Sigma^{0}(a,b)=\{(x,y)\in \mathbb{C}^2 \; | \; y^2=x^3+ax+b\}\,.
 \end{equation}
 The corresponding projective elliptic curve is denoted by $\Sigma(a,b)\subset \mathbb{P}^2$. The bundle $\Gamma\to M$ whose fibers are given by the first homology $\Gamma|_{(a,b)}=H_1(\Sigma(a,b),\mathbb{Z})$ has a natural flat structure, and picking a flat frame $\gamma_1$, $\gamma_2$ with intersection pairing $\gamma_1\cdot \gamma_2=1$, we can define coordinates $(z_1,z_2)$ on $M$
 by  
 \begin{equation*}
    z^i=\int_{\gamma_i}y\mathrm{d}x\,.  
 \end{equation*}
The coordinates $(z^1,z^2)$ are flat with respect to the flat structure induced by $\Gamma$, and the vector field $Z$ is given by \begin{equation*}
    Z=z^1\partial_{z^1}+z^2\partial_{z^2}\,.
\end{equation*} The symplectic structure $\Omega$ on $M$ is given by
\begin{equation}\label{sympconv}
    \Omega=-\mathrm{d}z^1\wedge \mathrm{d}z^2=2\pi \mathrm{i}\cdot \mathrm{d}a\wedge \mathrm{d}b\,.
\end{equation}
Finally, $h$ is the most interesting part and is built using the isomonodromy flows of the Painlev\'e I equation. In order to describe it explicitly in terms of an algebraic expression for the Pleba\'nski function, we consider several coordinates on $TM$.\\

On one hand, we have the natural induced coordinates $(z^1,z^2,\theta^1,\theta^2)$ on $TM$ induced from $(z^1,z^2)$, where the Joyce structure has the usual form for some Pleba\'nski function $W$. On the other hand,  in order to describe $W$ it is better to introduce the coordinates $(a,b,q,r)$. The relation between $(z^1,z^2,\theta^1,\theta^2)$ and $(a,b,q,r)$ is given as follows. First we introduce the periods $\omega_i$ and quasi-periods $\eta_i$ of $\Sigma(a,b)$ given by 
\begin{equation*}
\eta_i=-\frac{\partial z^i}{\partial a}, \quad \omega_i=\frac{\partial z^i}{\partial b}.
\end{equation*}
Note that they satisfy the relation $\omega_2\eta_1-\omega_1\eta_2=2\pi \mathrm{i}$. Next, if $(q,p)\in \Sigma^{0}(a,b)$ (so that $p^2=q^3+aq+b$), we define
\begin{equation*}
\theta_a:=-\frac{1}{4}\int_{(q,-p)}^{(q,p)}\frac{\mathrm{d}x}{y}, \quad \theta_b:=\frac{1}{4}\int_{(q,-p)}^{(q,p)}\frac{x\mathrm{d}x}{y}-r
\end{equation*}
where the integrals are along a path in $\Sigma^0(a,b)$ invariant under the involution $(x,y)\to (x,-y)$. The relation between $(\theta^1,\theta^2)$ and $(a,b,q,r)$ is then given by
\begin{equation*}
    \theta^i=-\eta_i\theta_a+\omega_i\theta_b\,.
\end{equation*}
 In the coordinates $(a,b,q,r)$ the Pleba\'nski function is then given by
\begin{equation}\label{A2pleb}
    \frac{1}{2\pi \mathrm{i}}W(a,b,q,r)=-\frac{2ap^2 +3pr(3b-2aq) + (6aq^2-9bq+4a^2)r^2 -2apr^3}{4p \Delta}, \quad \Delta=4a^3+27b^2
\end{equation}
but note that this $W$ does not satisfy the normalization conditions \eqref{plebzerosec}. For future reference, we record the relation between the coordinate vector fields on $TM$ induced by $(z^1,z^2,\theta^1,\theta^2)$ and $(a,b,q,r)$. They are given by 
\begin{equation}
\begin{split}
    \partial_{z^1}=-\frac{1}{2\pi \mathrm{i}}\left(\omega_2\partial_{a}+\eta_2\partial_b\right),& \quad \partial_{z^2}=\frac{1}{2\pi \mathrm{i}}\left(\omega_1\partial_{a}+\eta_1\partial_b\right)\\
    \partial_{\theta^1}=\frac{1}{2\pi \mathrm{i}}\left(\omega_2\left(2p \partial_q + q\partial_r\right)+\eta_2\partial_r\right), & \quad \partial_{\theta^2}=-\frac{1}{2\pi \mathrm{i}}\left(\omega_1\left(2p \partial_q + q\partial_r\right)+\eta_1\partial_r\right)\\
\end{split}
\end{equation}
On the other hand, the relation between $(a,b,q,r)$ and $(a,b,\theta_a,\theta_b)$ is given by
\begin{equation}\label{usefulid}
    \partial_{\theta_a}=-2p\partial_{q}-q\partial_r, \quad \partial_{\theta_b}=-\partial_r\,.
\end{equation}
and 
\begin{equation}\label{usefulid2}
    \partial_q=-\frac{1}{2p}\partial_{\theta_a}+\frac{q}{2p}\partial_{\theta_b}\,.
\end{equation}

\subsubsection{Computing $\dot{g}_2$}

Recall from Remark \ref{g1hamrem}  that if we wish to solve the equations from Proposition \ref{receq} for $\dot{g}=\sum_{k=1}^{\infty}\dot{g}_k\epsilon^k$, we can take $\dot{g}_1=-\text{Ham}^{\Omega^{v}}(W)$, where $W$ in the A2 case is given by \eqref{A2pleb}. We will now explicitly compute $W_2$ giving $g_2=\text{Ham}^{\Omega^{v}}({W_2})$.\\

Suppose that we pick $\nabla$-flat Darboux coordinates $(Z^1,Z^2)$ on $M$ (not necessarily equal to the previous $(z^1,z^2)$ given before), where
\begin{equation*}
\Omega=\Omega_{ij}\mathrm{d}Z^i\wedge \mathrm{d}Z^j
\end{equation*}
for some constant matrix $\Omega_{ij}$. With respect to the induced coordinates $(Z^i,\theta^i)$ on $TM$, the equation that we wish to solve is (recall \eqref{firstreceq}) 

\begin{equation}\label{maing2eq}
        [\partial_{\theta^i},\dot{g}_2]=[\dot{g}_1,\partial_{Z^i}]+\frac{1}{2}[\dot{g_1},[\dot{g_1},\partial_{\theta^i}]].
\end{equation}
Suppose now that $W^i$ are Darboux coordinates on $M$ that are not $\nabla$-flat. Namely, we can still write
\begin{equation*}
    \Omega=\widetilde{\Omega}_{ij}\mathrm{d}W^i\wedge \mathrm{d}W^j
\end{equation*}
for some other constant matrix $\widetilde{\Omega}_{ij}$ but $\nabla (\mathrm{d}W^i)\neq 0$. In the induced coordinates $(W^i,\phi^i)$, on $TM$, the equation \eqref{maing2eq} is now rewritten as
\begin{equation}\label{g2eqothercoords}
        [\partial_{\phi^i},\dot{g}_2]=[\dot{g}_1,\mathcal{H}_{\partial_{W^i}}]+\frac{1}{2}[\dot{g_1},[\dot{g_1},\partial_{\phi^i}]]\,,
\end{equation}
where now $\mathcal{H}_{\partial_{W^i}}\neq \partial_{W^i}$. Indeed, while the $Z^i$ are flat coordinates for $\mathcal{H}$, so that $\mathcal{H}_{\partial_{Z^i}}=\partial_{Z^i}$, the $W^i$ are not flat and instead we have
\begin{equation}
    \mathcal{H}_{\partial_{W^i}}=\partial_{W^i}+f_i^q(W,\phi)\partial_{\phi^q}\,.
\end{equation}
In the coordinates $(W^1,W^2,\phi^1,\phi^2)=(a,b,\theta_a,\theta_b)$ the functions $f(W,\phi)$ can be computed using similar computations to \cite[Section 4.2]{BridgelandOsc}, where we find that 
\begin{equation}\label{horliftid}
    \mathcal{H}_{\partial_{W^i}}=\partial_{W^i}+\widetilde{\Omega}^{jq}C_{ijk}(w)\phi^k\partial_{\phi^q}\,
\end{equation}
where $\widetilde{\Omega}^{ij}$ is the inverse matrix of $\widetilde{\Omega}_{ij}$ and $C_{ijk}(W)$ are the coefficients of the cubic polynomial from \cite[Equation (66)]{BridgelandOsc} given by
\begin{equation}\label{cubicpol}
\frac{1}{2\pi \mathrm{i}}C(W,\phi)=\frac{1}{6}C_{ijk}(W)\phi^i\phi^j\phi^k=\frac{1}{4\Delta}\left(ab\theta_a^3-2a^2\theta_a^2\theta_b-9b\theta_a\theta_b^2+2a\theta_b^3\right)\,.
\end{equation}
With our choice of symplectic form \eqref{sympconv} and \eqref{cubicpol} and renaming the indices to $(1,2)=(a,b)$ we have
\begin{equation}\label{usefulid3}
    \widetilde{\Omega}^{ab}=-\frac{1}{2\pi \mathrm{i}}, \quad \frac{1}{2\pi \mathrm{i}}C_{aaa}=\frac{3ab}{2\Delta}, \quad \frac{1}{2\pi \mathrm{i}}C_{aab}=-\frac{a^2}{\Delta}, \quad \frac{1}{2\pi \mathrm{i}}C_{abb}=-\frac{9b}{2\Delta}, \quad \frac{1}{2\pi \mathrm{i}}C_{bbb}=\frac{3a}{\Delta}\,
\end{equation}
where we used that $C_{ijk}$ is fully symmetric. \\

We now have all the ingredients to substitute into equation \eqref{g2eqothercoords} and find $W_2$ such that $\dot{g}_2=\text{Ham}^{\Omega^{v}}(W_2)$ satisfies \eqref{maing2eq}. Before we do so, let us motivate a bit the ansatz that we use to find $W_2$.\\

Recall (J4) from Definition \ref{Joyce_def} where we consider the involution $\iota:TM\to TM$ acting by $-1$ on the fibres of $\pi:TM\to M$. There are two natural ways to lift the involution to $\iota:TM\times \mathbb{C}\to TM\times \mathbb{C}$ by acting by $\pm 1$ on the $\mathbb{C}$ factor. The choice of $\iota(X,\epsilon)=(-X,-\epsilon)$ leaves the relative connection $\mathcal{A}$ of the Joyce structure invariant, so we make the $-1$ choice. From the equation $e^{\mathcal{L}_{\dot{g}}}\mathcal{A}^{\text{st}}=\mathcal{A}$ it is then easy to see that $\dot{g}$ must be invariant under $\iota$, and hence the components in the $\epsilon$-expansion satisfy that  $\dot{g}_{2n}$ (resp. $\dot{g}_{2n+1}$) is invariant (resp. odd) under the involution.\\

On the other hand, note that in the $(a,b,q,r)$ coordinates, the formula for $W$ in \eqref{A2pleb} has the form
\begin{equation*}
    W=\sum_{k=0}^{3}W_{k}(a,b,q)r^k
\end{equation*}
so it is natural to assume a similar form 
\begin{equation*}
    W_2=\sum_{k=0}^mW_{2,k}(a,b,q)r^k
\end{equation*}
The involution $\iota(z^1,z^2,\theta^1,\theta^2)=(z^1,z^2,-\theta^1,-\theta^2)$ translates into \cite[Remark 3.5]{BridgelandOsc}
\begin{equation*}
    (a,b,q,p,r)\to (a,b,q,-p,-r), \quad \text{where} \quad p^2=q^3+aq+b\,,
\end{equation*}
so if we look for a $W_2$ invariant under $\iota$ (so that $\dot{g}_2$ is invariant), it follows that $W_{2,k}$ is invariant for even $k$ and odd for odd $k$. With this in mind, we now show:

\begin{proposition}\label{w2prop}
    There exists $W_2$ invariant under the involution $\iota$ such that
    \begin{equation*}
        \dot{g}_2=\text{Ham}^{\Omega^{v}}(W_2)
    \end{equation*}
    satisfies \eqref{g2eqothercoords}. In the coordinates $(a,b,q,r)$ it has the form
    \begin{equation}\label{w2ansatz}
        \frac{1}{2\pi \mathrm{i}}W_2(a,b,q,r)=\sum_{k=0}^{4}W_{2,i}(a,b,q)r^k
    \end{equation}
    where 
    \begin{equation}\label{w2compexplicit}
        \begin{split}
        W_{2,0}&=\frac{27abq^2}{4\Delta^2}+\frac{11q}{32\Delta}-\frac{3a^3q}{\Delta^2}\\
        W_{2,1}&=\frac{a}{8\Delta p}-\frac{27ab p}{\Delta^2}\\
        W_{2,2}&=\frac{1}{16\Delta p^2}\left(-3aq+\frac{9b}{2}\right)+\frac{3\cdot 27abq}{2\Delta^2}+\frac{3\cdot 27(2a^3-11b^2)}{32\Delta^2}\\
        W_{2,3}&=\frac{1}{12\Delta p^3}\left(\frac{3aq^2}{2}-\frac{9bq}{4}+a^2\right)+\frac{1}{\Delta^2p}\left(-27abq^2+\left(-\frac{13\Delta}{16}+\frac{81b^2}{2}\right)q-18a^2b\right)\\
        W_{2,4}&=\frac{3^4ab}{2^4\Delta^2}-\frac{3a}{2^4\Delta p^2}\\
    \end{split}
    \end{equation}
\end{proposition}
\begin{proof}
Due to the long and complicated expressions that appear, we leave the details of the proof in the Appendix \ref{A1}.
\end{proof}

\begin{remark}
\begin{itemize}
\item Note that \eqref{w2compexplicit} and \eqref{w2ansatz} show that $W_2$ is an algebraic function of $(a,b,q,r)$, which is also true for the Pleba\'nski function $W$. 
\item One could use $\dot{g}_1=-\text{Ham}^{\Omega^{v}}(W)$ and $\dot{g}_2=\text{Ham}^{\Omega^{v}}(W_2)$ to compute expansions of the formal twistor coordinates $(x^1,x^2)$ for the Joyce structure of the A2 quiver up to $\epsilon^2$-order. It would be interesting to check if they provide an asymptotic approximation as $\epsilon \to 0$ of the analytic twistor coordinates coming from the solution of the associated Riemann-Hilbert problem in \cite{BridgelandOsc}.
\item In \cite[Section 8.3]{BridgelandTau} it is mentioned that a certain restriction of the $\tau$-function associated to the Joyce structure of the A2 quiver coincides with the Painlev\'e I $\tau$-function. In \cite[Equation 8.4]{BridgelandTau} the formula for the Joyce structure $\tau$-function is given in terms of the 1-form $x^1\mathrm{d}x^2$, where the exterior derivative does not differentiate in $\epsilon$. Hence, one could use the previous expansion of the $x^i$ up to $\epsilon^2$-order to compute an expansion of the $\tau$-function. 
\end{itemize}
\end{remark}

\appendix

\section{Details on the computation of $\dot{g}_2$}\label{A1}

In this appendix with prove Proposition \ref{w2prop}. Throughout the computation we make extensive use of \eqref{usefulid}, \eqref{usefulid2}, and \eqref{usefulid3} and switch back and forth between the coordinates $(a,b,\theta_a,\theta_b)$ and $(a,b,q,r)$. \\

Using that $\dot{g}_1=-\text{Ham}^{\Omega^{v}}(W)$, assuming that $\dot{g}_2$ has Hamiltonian form $\dot{g}_2=\text{Ham}^{\Omega^{v}}(W_2)$, and using \eqref{horliftid} with $(W^1,W^2,\phi^1,\phi^2)=(a,b,\theta_a,\theta_b)$, we can rewrite \eqref{maing2eq} evaluated on $\partial_{W^2}=\partial_b$ and $\partial_{\phi^2}=\partial_{\theta_b}$ as follows

\begin{equation}\label{w2eq1}
    \begin{split}
        \frac{1}{2\pi \mathrm{i}}&\left(\frac{\partial^2W_2}{\partial\theta_b^2}\frac{\partial}{\partial \theta_a}-\frac{\partial^2 W_2}{\partial \theta_a\partial\theta_b}\frac{\partial}{\partial \theta_b}\right)\\
        =&\frac{1}{2\pi \mathrm{i}}\left(\frac{\partial^2W}{\partial b\partial\theta_b}\frac{\partial}{\partial \theta_a}-\frac{\partial^2 W}{\partial b\partial \theta_a}\frac{\partial}{\partial \theta_b}\right)\\
        &+\frac{1}{(2\pi\mathrm{i})^2}\left(\frac{\partial W}{\partial \theta_a}C_{bbb}-\frac{\partial W}{\partial \theta_b}C_{bba}+C_{bbk}\theta^k\frac{\partial^2W}{\partial \theta_a\partial \theta_b}-C_{bak}\theta^k\frac{\partial^2W}{\partial \theta_b^2}\right)\frac{\partial}{\partial \theta_a}\\
        &+\frac{1}{(2\pi\mathrm{i})^2}\left(-\frac{\partial W}{\partial \theta_a}C_{abb}+\frac{\partial W}{\partial \theta_b}C_{aab}+C_{abk}\theta^k\frac{\partial^2W}{\partial \theta_a\partial \theta_b}-C_{bbk}\theta^k\frac{\partial^2W}{\partial \theta_a^2}\right)\frac{\partial}{\partial \theta_b}\\
        &-\frac{1}{2}\frac{1}{(2\pi\mathrm{i})^2}\left(\frac{\partial W}{\partial \theta_{b}}\frac{\partial^3 W}{\partial \theta_a \partial \theta_b^2}- \frac{\partial W}{\partial \theta_{a}}\frac{\partial^3 W}{ \partial \theta_b^3}\right)\frac{\partial}{\partial \theta_a}\\
    &-\frac{1}{2}\frac{1}{(2\pi\mathrm{i})^2}\left(\frac{\partial W}{\partial \theta_a}\frac{\partial^3W}{\partial \theta_a \partial \theta_b^2}- \frac{\partial W}{\partial \theta_b}\frac{\partial^3W}{\partial \theta_a^2 \partial \theta_b} + \frac{\partial^2W}{\partial \theta_a^2}\frac{\partial^2W}{\partial \theta_b^2}-\left(\frac{\partial^2W}{\partial \theta_a \partial \theta_b}\right)^2\right)\frac{\partial}{\partial \theta_b}\\
    \end{split}
\end{equation}
Looking at the $\partial_{\theta_a}$ component in \eqref{w2eq1}, we obtain

\begin{equation}\label{w2eq2}
    \begin{split}
        \frac{1}{2\pi \mathrm{i}}\frac{\partial^2 W_2}{\partial \theta_b^2}=&\frac{1}{2\pi \mathrm{i}}\frac{\partial^2 W}{\partial b \partial \theta_b}+\frac{1}{(2\pi\mathrm{i})^2}\left(\frac{\partial W}{\partial \theta_a}C_{bbb}-\frac{\partial W}{\partial \theta_b}C_{bba}+C_{bbk}\theta^k\frac{\partial^2W}{\partial \theta_a\partial \theta_b}-C_{bak}\theta^k\frac{\partial^2W}{\partial \theta_b^2}\right)\\
        &-\frac{1}{2}\frac{1}{(2\pi\mathrm{i})^2}\left(\frac{\partial W}{\partial \theta_{b}}\frac{\partial^3 W}{\partial \theta_a \partial \theta_b^2}- \frac{\partial W}{\partial \theta_{a}}\frac{\partial^3 W}{ \partial \theta_b^3}\right)\,.
    \end{split}
\end{equation}
Taking two extra derivatives in $\theta_b$ of \eqref{w2eq2} we obtain

\begin{equation}\label{w2eq4}
    \begin{split}
        \frac{1}{2\pi \mathrm{i}}\frac{\partial^4 W_2}{\partial \theta_b^4}&=\frac{1}{2\pi \mathrm{i}}\frac{\partial^4 W}{\partial b \partial \theta_b^3}+\frac{1}{(2\pi\mathrm{i})^2}\left(-3\frac{\partial^3W}{\partial \theta_b^3}C_{bba}+3\frac{\partial^3 W}{\partial \theta_b^2 \partial \theta_a}C_{bbb}\right)\,.
    \end{split}
\end{equation}
so substituting \eqref{usefulid3} in \eqref{w2eq4} we want to solve

\begin{equation}\label{w2eq5}
    \begin{split}
        \frac{\partial^4 W_2}{\partial \theta_b^4}&=\frac{\partial^4 W}{\partial b \partial \theta_b^3}+\left(\frac{27b}{2\Delta}\frac{\partial^3W}{\partial \theta_b^3}+\frac{9a}{\Delta}\frac{\partial^3 W}{\partial \theta_b^2 \partial \theta_a}\right)\,.\\
    \end{split}
\end{equation}
Using the ansatz \eqref{w2ansatz} (note that the highest power of $r$ is $4$) and \eqref{usefulid}
we find that \eqref{w2eq5} reduces to
\begin{equation*}
    24W_{2,4}=-3\partial_b\frac{a}{\Delta}-\frac{3\cdot 27ab}{2\Delta^2 }-\frac{9a}{2\Delta p^2}=\frac{9\cdot 27ab}{2\Delta^2}-\frac{9a}{2\Delta p^2}
\end{equation*}
and hence 
\begin{equation*}
    W_{2,4}=\frac{3^4ab}{2^4\Delta^2}-\frac{3a}{2^4\Delta p^2}\,.
\end{equation*}
Taking now a derivative in $\theta_b$ and $\theta_a$ in \eqref{w2eq2}  we find that 
\begin{equation}\label{w2eq6}
    \begin{split}
        \frac{1}{2\pi \mathrm{i}}\frac{\partial^4 W_2}{\partial \theta_a\partial \theta_b^3}=&\frac{1}{2\pi \mathrm{i}}\frac{\partial^4 W}{\partial b \partial \theta_a \partial \theta_b^2}+\frac{1}{(2\pi\mathrm{i})^2}\left(2\frac{\partial^3W}{\partial \theta_a^2 \partial \theta_b}C_{bbb}-\frac{\partial^3W}{\partial \theta_a\partial \theta_b^2}C_{bba}-\frac{\partial^3 W}{\partial \theta_b^3}C_{baa}+C_{bbk}\theta^k\frac{\partial^2W}{\partial \theta_a^2\partial \theta_b^2}\right)\\
        &-\frac{1}{2}\frac{1}{(2\pi\mathrm{i})^2}\left(\left(\frac{\partial^3W}{\partial \theta_a \partial\theta_b^2}\right)^2+\frac{\partial^2W}{\partial \theta_b^2}\frac{\partial^4W}{\partial \theta_a^2 \partial\theta_b^2}-\frac{\partial^3W}{\partial \theta_a^2 \partial\theta_b}\frac{\partial^3W}{ \partial\theta_b^3}\right)\,.
    \end{split}
\end{equation}
Using \eqref{usefulid} and our previous $W_{2,4}$ we obtain a differential equation for $W_{2,3}$ having a unique solution odd under $\iota$ given by 
\begin{equation*}
\begin{split}
    W_{2,3}&=\frac{1}{12\Delta p^3}\left(\frac{3aq^2}{2}-\frac{9bq}{4}+a^2\right)+\frac{1}{\Delta^2p}\left(-27abq^2+\left(-\frac{13\Delta}{16}+\frac{81b^2}{2}\right)q-18a^2b\right)\,.
\end{split}
\end{equation*}
Continuing in the same fashion as before, taking two $\theta_a$ derivatives of \eqref{w2eq2} we find 

\begin{equation}\label{w2eq7}
    \begin{split}
        \frac{1}{2\pi \mathrm{i}}&\frac{\partial^4 W_2}{\partial \theta_a^2\partial \theta_b^2}=\\
        &\frac{1}{2\pi \mathrm{i}}\frac{\partial^4 W}{\partial b \partial \theta_a^2 \partial \theta_b}+\frac{1}{(2\pi\mathrm{i})^2}\left(\frac{\partial^3W}{\partial \theta_a^3}C_{bbb}+\frac{\partial^3W}{\partial \theta_a^2\partial \theta_b}C_{abb}+C_{bbk}\theta^k\frac{\partial^4W}{\partial \theta_a^3 \partial \theta_b}-2\frac{\partial^3 W}{\partial \theta_a\partial \theta_b^2}C_{baa}-C_{bak}\theta^k\frac{\partial^2W}{\partial \theta_a^2\partial \theta_b^2}\right)\\
        &-\frac{1}{2}\frac{1}{(2\pi\mathrm{i})^2}\left(\frac{\partial^3W}{\partial \theta_a^2 \partial\theta_b}\frac{\partial^3W}{\partial \theta_a \partial\theta_b^2}+2\frac{\partial^2W}{\partial \theta_a\partial \theta_b}\frac{\partial^4W}{\partial \theta_a^2 \partial\theta_b^2}+\frac{\partial W}{\partial \theta_b}\frac{\partial^5W}{\partial \theta_a^3 \partial\theta_b^2}-\frac{\partial^3W}{\partial \theta_a^3}\frac{\partial^3W}{ \partial\theta_b^3}\right)\,.
    \end{split}
\end{equation}
Using \eqref{usefulid} and our previous $W_{2,3}$ and $W_{2,4}$ we can simplify this to give the differential equation
\begin{equation}\label{a2eq8}
    (2p\partial_q)^2(2W_{2,2})=2p\partial_q\left(-\frac{1}{2\Delta p^3}\left(\frac{27b}{4}q^2+3a^2q+\frac{27ab}{4}\right)+\frac{3a}{2\Delta p}+\frac{6\cdot 27abp}{\Delta^2}\right)\,.
\end{equation}
Using that $2p\partial_qW_{2,2}$ is odd under the involution we can integrate \eqref{a2eq8} once to the unique solution
\begin{equation*}
    (2p\partial_q)(2W_{2,2})=-\frac{1}{2\Delta p^3}\left(\frac{27b}{4}q^2+3a^2q+\frac{27ab}{4}\right)+\frac{3a}{2\Delta p}+\frac{6\cdot 27abp}{\Delta^2}
\end{equation*}
To continue to integrate, we use that 
\begin{equation*}
    2p\partial_q\left(\frac{\alpha q^2+\beta q +\gamma}{p^2}\right)=\frac{(4a\alpha-6\gamma)q^2+(6\alpha b+4a\beta)q+(6b\beta-2\gamma a)}{p^3}-\frac{2\alpha q+4\beta}{p}\,,
\end{equation*}
and find that 
\begin{equation*}
    W_{2,2}=\frac{1}{16\Delta p^2}\left(-3aq+\frac{9b}{2}\right)+\frac{3\cdot 27abq}{2\Delta^2}+c(a,b)\,,
\end{equation*}
for some unknown function $c(a,b)$ depending only on $(a,b)$.\\

Continuing with the computation, to find $W_{2,1}$ we now consider the $\partial_{\theta_b}$ component of \eqref{w2eq1} and take two $\theta_a$ derivatives to obtain
\begin{equation}\label{w2eq9}
    \begin{split}
        \frac{\partial^4W_{2}}{\partial \theta_a^3\partial \theta_b}=&\frac{\partial^4W}{\partial b \partial \theta_a^3}-\left(-3C_{abb}\frac{\partial^3W}{\partial \theta_a^3}+3C_{aab}\frac{\partial^3W}{\partial \theta_a^2 \partial \theta_b}+C_{abk}\theta^k\frac{\partial^4W}{\partial \theta_a^3 \partial \theta_b}-C_{bbk}\theta^k\frac{\partial^4W}{\partial \theta_a^4}\right)\\
        &+\frac{1}{2}\left(3\frac{\partial^3W}{\partial \theta_a^3}\frac{\partial^3W}{\partial \theta_a \partial \theta_b^2}+3\frac{\partial^2W}{\partial \theta_a^2}\frac{\partial^4W}{\partial \theta_a^2 \partial \theta_b^2}+\frac{\partial W}{\partial \theta_a}\frac{\partial^5W}{\partial \theta_a^3 \partial \theta_b^2}-3\left(\frac{\partial^3W}{\partial \theta_a^2 \partial \theta_b}\right)^2\right)\\
        &+\frac{1}{2}\left(-4\frac{\partial^2W}{\partial \theta_a \partial \theta_b}\frac{\partial^4W}{\partial \theta_a^3 \partial \theta_b}-\frac{\partial W}{\partial \theta_b}\frac{\partial^5W}{\partial \theta_a^4 \partial \theta_b}+\frac{\partial^4W}{\partial \theta_a^4}\frac{\partial^2W}{\partial \theta_b^2}\right)\\
    \end{split}
\end{equation}
If we choose\footnote{This choice is made so that we can easily integrate the resulting equation.}
\begin{equation}
c(a,b)=\frac{3\cdot 27(2a^3-11b^2)}{32\Delta^2}
\end{equation}
and use \eqref{usefulid}, $W_{2,4}$, $W_{2,3}$ and $W_{2,2}$ from before, we can then simplify \eqref{w2eq9} to the differential equation
\begin{equation}\label{w2eq10}
    (2p\partial_q)^3W_{2,1}=2p\partial_q\left(\frac{1}{4\Delta p^3}(-3a^2q^2-9abq+a^3)+\frac{3aq}{4\Delta p}-\frac{12\cdot 27 abpq}{\Delta^2}\right)\,.
\end{equation}
Using that $(2p\partial_q)^2W_{2,1}$ is odd under the involution we can integrate once to find the unique odd solution
\begin{equation*}
\begin{split} (2p\partial_q)^2W_{2,1}&=\frac{1}{4\Delta p^3}(-3a^2q^2-9abq+a^3)+\frac{3aq}{4\Delta p}-\frac{12\cdot 27 abpq}{\Delta^2}\\
&=2p\partial_q\left(-\frac{3aq^2+a^2}{8\Delta p^2}-\frac{3\cdot 27 abq^2}{\Delta^2}\right)\,.
\end{split}
\end{equation*}
Hence, integrating once more we find that for some function $d(a,b)$ of $(a,b)$
\begin{equation*}
    2p\partial_qW_{2,1}=-\frac{3aq^2+a^2}{8\Delta p^2}-\frac{3\cdot 27 abq^2}{\Delta^2}+d(a,b)\,.
\end{equation*}
We set $d(a,b)=-27a^2b/\Delta^2$ so that using again that $W_{2,1}$ is odd under $\iota $ we can integrate and find 
\begin{equation*}
    W_{2,1}=\frac{a}{8\Delta p}-\frac{27ab p}{\Delta^2}\,.
\end{equation*}
For the last part, we start from the equation 
\begin{equation}\label{w2eq11}
\begin{split}
\frac{1}{2\pi \mathrm{i}}\frac{\partial^2 W_2}{\partial \theta_a^2}=&\frac{1}{2\pi \mathrm{i}}\frac{\partial^2W}{\partial a \partial \theta_a}+\frac{1}{(2\pi \mathrm{i})^2}\left(-C_{aaa}\frac{\partial W}{\partial  \theta_b}+C_{aab}\frac{\partial W}{\partial \theta_a}+C_{abk}\theta^k\frac{\partial^2 W}{\partial \theta_a^2}-C_{aak}\theta^k\frac{\partial^2W}{\partial \theta_a \partial \theta_b}\right)\\
&+\frac{1}{2(2\pi \mathrm{i})^2}\left(\frac{\partial W}{\partial \theta_a}\frac{\partial^3W}{\partial \theta_a^2 \partial \theta_b}-\frac{\partial W}{\partial \theta_b}\frac{\partial^3 W}{\partial \theta_a^3}\right)\,
\end{split}
\end{equation}
obtained from the $\partial_{\theta_b}$ component of the evaluation of \eqref{maing2eq} on $\partial_{W^1}=\partial_a$, $\partial_{\phi^1}=\partial_{\theta_a}$. Although we will not write the details, we remark that the equation for $\frac{\partial^2W_2}{\partial \theta_a\partial\theta_b}$ obtained from the $\partial_{\theta_a}$ component is compatible with the one obtained from the $\partial_{\theta_b}$ component of \eqref{w2eq1} due to the fact that $W$ satisfies the Pleba\'nski second heavenly equations in $(Z^1,Z^2,\theta^1,\theta^2)$ coordinates. \\

Taking two derivatives of \eqref{w2eq11} in $\theta_a$ we obtain

\begin{equation}\label{w2eq12}
\begin{split}
\frac{1}{2\pi \mathrm{i}}\frac{\partial^2 W_2}{\partial \theta_a^4}=&\frac{1}{2\pi \mathrm{i}}\frac{\partial^2W}{\partial a \partial \theta_a^3}+\frac{1}{(2\pi \mathrm{i})^2}\left(-3C_{aaa}\frac{\partial^3 W}{\partial \theta_a^2\partial  \theta_b}+3C_{aab}\frac{\partial^2 W}{\partial \theta_a^3}+C_{abk}\theta^k\frac{\partial^4 W}{\partial \theta_a^4}-C_{aak}\theta^k\frac{\partial^4W}{\partial \theta_a^3 \partial \theta_b}\right)\\
&+\frac{1}{2(2\pi \mathrm{i})^2}\left(\frac{\partial W}{\partial \theta_a}\frac{\partial^5W}{\partial \theta_a^4 \partial \theta_b}-\frac{\partial W}{\partial \theta_b}\frac{\partial^5W}{\partial \theta_a^5}+2\frac{\partial^2W}{\partial \theta_a^2}\frac{\partial^4W}{\partial \theta_a^3 \partial \theta_b}-2\frac{\partial^2W}{\partial \theta_a \partial \theta_b}\frac{\partial^4W}{\partial \theta_a^4}\right)\,.
\end{split}
\end{equation}
Using \eqref{usefulid} and the previous results, we can rewrite the resulting equation as 
\begin{equation}\label{w2eq13}
(2p\partial_q)^4W_{2,0}=2p\partial_q\left(\frac{162a^2bp}{\Delta^2}+\frac{30\cdot 27abq^2p}{\Delta^2}+\frac{-3\cdot 13\cdot 4a^3+3\cdot 27\cdot 11 b^2}{4\Delta^2}qp\right)\,.
\end{equation}
Using the oddness of $(2p\partial_q^3)W_{2,0}$ can integrate \eqref{w2eq13} uniquely to find 

\begin{equation*}
(2p\partial_q)^3W_{2,0}=\frac{162a^2bp}{\Delta^2}+\frac{30\cdot 27abq^2p}{\Delta^2}+\frac{-3\cdot 13\cdot 4a^3+3\cdot 27\cdot 11 b^2}{4\Delta^2}qp
\end{equation*}
which can be further integrated to 
\begin{equation}
\begin{split}(2p\partial_q)^2W_{2,0}=2p\partial_q\left(\frac{27abqp}{\Delta^2}\right)-\frac{2\cdot 27ab^2}{\Delta^2}+\frac{-3\cdot 13\cdot 4a^3+3\cdot 27\cdot 11 b^2}{16\Delta^2}q^2+e(a,b)
\end{split}
\end{equation}
for some function $e(a,b)$. We choose 
\begin{equation*}
    e(a,b)=\frac{2\cdot 27ab^2}{\Delta^2}+\frac{- 13\cdot 4a^4+ 27\cdot 11 ab^2}{16\Delta^2}
\end{equation*}
so that 
\begin{equation}\label{w2eq14}
    (2p\partial_q)^2W_{2,0}=2p\partial_q\left(\frac{27abqp}{\Delta^2}+\frac{- 13\cdot 4a^3+ 27\cdot 11 b^2}{16\Delta^2}p\right)
\end{equation}
Using the fact that  $2p\partial_qW_{2,0}$ is odd under $\iota$ we can integrate \eqref{w2eq14} uniquely to  

\begin{equation}\label{a2eq15}
    (2p\partial_q)W_{2,0}=\frac{27abqp}{\Delta^2}+\frac{- 13\cdot 4a^3+ 27\cdot 11 b^2}{16\Delta^2}p\,.
\end{equation}
Finally, the latter can be integrated (up to a function of $(a,b)$) to 
\begin{equation*}
\begin{split}
    W_{2,0}=\frac{27abq^2}{4\Delta^2}+\frac{11q}{32\Delta}-\frac{3a^3q}{\Delta^2}
\end{split}
\end{equation*}
Hence, joining everything together we find \eqref{w2compexplicit}.

\bibliography{References}   
\bibliographystyle{alpha}
\end{document}